\documentclass[12pt]{article}
\usepackage{amsmath,amssymb,amsthm,graphicx,color,bbm}
\usepackage[left=1in,right=1in,top=1in,bottom=1in]{geometry}
\usepackage[british]{babel}
\usepackage{hyperref} 
\usepackage{enumitem}
\usepackage{csquotes,caption}
\usepackage{subcaption}

\usepackage[backend=bibtex,citestyle=numeric,doi=false,isbn=false,url=false,eprint=false,uniquename=full,maxbibnames=10,sortcites]{biblatex}

\DeclareNameFormat{always-init}{%
\ifnumequal{\value{uniquename}}{2}
    {\usebibmacro{name:family-given}
      {\namepartfamily}
      {\namepartgiven}
      {\namepartprefix}
      {\namepartsuffix}}
    {\usebibmacro{name:family-given}
      {\namepartfamily}
      {\namepartgiveni}
      {\namepartprefix}
      {\namepartsuffix}}%
  \usebibmacro{name:andothers}}

\DeclareNameAlias{author}{always-init}
\DeclareNameAlias{editor}{always-init}

\DeclareSortingNamekeyTemplate{
  \keypart{\namepart{family}}
  \keypart{\namepart{prefix}}
  \keypart{\namepart{given}}
  \keypart{\namepart{suffix}}}

\AtEveryBibitem{\clearfield{number}}
\AtEveryBibitem{\clearfield{series}}

\DeclareFieldFormat[article,inbook,incollection,inproceedings,patent,thesis,unpublished]{title}{{#1\isdot}}
  \renewbibmacro{in:}{%
  \ifentrytype{article}{}{\printtext{\bibstring{in}\intitlepunct}}}

\newtheorem{theorem}{Theorem}[section]
\newtheorem{lemma}[theorem]{Lemma}
\newtheorem{proposition}[theorem]{Proposition}

\theoremstyle{definition}
\newtheorem{definition}[theorem]{Definition}

\newtheorem{problem}[theorem]{Problem}

\theoremstyle{remark}
\newtheorem{remark}[theorem]{Remark}
\newtheorem{remarks}[theorem]{Remarks}

\numberwithin{equation}{section}

\let\phi=\varphi

\allowdisplaybreaks

\newcommand{\1}[1]{{\mathbbm{1}\mkern -1.5mu}{\{#1\}}}
\newcommand{\2}[1]{{\mathbbm{1}}_{#1}}

\newcommand{\R}{{\mathbb R}}

\newcommand{\Z}{{\mathbb Z}}
\newcommand{\N}{{\mathbb N}}

\newcommand{\ZP}{{\mathbb Z}_+}
\newcommand{\RP}{{\mathbb R}_+}

\newcommand{\geo}[1]{\mathrm{Geom}_0 \left( {#1} \right)}

\newcommand{\eps}{\varepsilon}

\newcommand{\re}{{\mathrm{e}}}
\newcommand{\ud}{{\mathrm d}}

\newcommand{\cF}{{\mathcal F}}

\newcommand{\cI}{{\mathcal I}}

\newcommand{\cV}{{\mathcal V}}
\newcommand{\cW}{{\mathcal W}}

\newcommand{\as}{\ \text{a.s.}}

\newcommand{\IP}{{\mathbb P}}
\newcommand{\IE}{{\mathbb E}}
\newcommand{\tIP}{\widetilde{\mathbb P}}
\newcommand{\0}{{\mathbf 0}}

\newcommand{\bbX}{{\mathbb X}}
\newcommand{\bbXB}{{\mathbb X}_{\mathrm{F}}}
\newcommand{\bbD}{{\mathbb D}}
\newcommand{\bbDB}{{\mathbb D}_{\mathrm{F}}}

\makeatletter
\def\namedlabel#1#2{\begingroup  
    (#2)%
    \def\@currentlabel{#2}%
    \phantomsection\label{#1}\endgroup
}
\makeatother

\newlist{myenumi}{enumerate}{10}
\setlist[myenumi]{leftmargin=0pt, labelindent=\parindent, listparindent=\parindent, labelwidth=0pt, itemindent=!, itemsep=1pt, parsep=4pt}

\newlist{thmenumi}{enumerate}{10}
\setlist[thmenumi]{leftmargin=0pt, labelindent=\parindent, listparindent=\parindent, labelwidth=0pt, itemindent=!}

\makeatletter 
\@mparswitchfalse%
\makeatother
\reversemarginpar
\usepackage[colorinlistoftodos]{todonotes}

\title{Dynamics of the leftmost  particle in heterogeneous 
semi-infinite exclusion systems}

\author{Mikhail Menshikov\footnote{Mikhail Vasilyevich Menshikov (January 17, 1948--May 1, 2026) passed away after the first version of this paper was submitted. He will be deeply missed by both co-authors.} \and Serguei Popov \and Andrew Wade}

\date{\today}

\begin{document}
\maketitle

\begin{abstract}
We study the behaviour of the leftmost particle 
in a  semi-infinite particle system on~$\Z$,
where each particle performs a continuous-time nearest-neighbour random walk,
with particle-specific jump rates, subject to the exclusion interaction (i.e., no more than one particle per site).
We give conditions, in terms of the jump rates of the system, under which the leftmost particle is recurrent or transient, and develop tools to study its rate of escape in the transient case, including by comparison with an $M/G/\infty$ queue. In particular we 
show examples in which 
the leftmost particle can be null recurrent, positive recurrent, ballistically transient, or subdiffusively transient. Finally we indicate the
role of the initial condition in determining the dynamics, and show, for example,
that sub-ballistic transience can occur
started from close-packed initial configurations but not from stationary initial conditions.
\end{abstract}

\medskip


\medskip


\section{Introduction}
\label{sec:intro}

\subsection{Dynamics of the leftmost particle}
 \label{sec:intro-dynamics}
 
Consider a 
 {semi-infinite}
collection of particles living on
distinct sites of~$\Z$, with the particles  enumerated by $\N :=\{1,2,3,\ldots\}$
 from left to right; in particular, there is a 
leftmost particle, but no rightmost one.
The particles perform continuous-time, nearest-neighbour random walks with exclusion interaction (i.e., there can be no more than one particle at a given site),
in which each particle possesses arbitrary finite positive jump rates. Since the jumps that would lead to  violation of the exclusion rule are suppressed, the  order of the particles is 
preserved by the dynamics. 

This model
is an example of the famous \emph{exclusion process},
and we studied basic properties of semi-infinite systems with non-homogeneous particle jump rates in
our earlier paper~\cite{MPW25}, to which the present paper is a sequel. In~\cite{MPW25} we established conditions for \emph{stability} of the system,
started from initial configurations that are finite perturbations of the close-packed configuration, in which there are
no empty sites between successive particles. In the stable situation, finite-dimensional inter-particle distances converge to product-geometric stationary distributions, and each particle in the system satisfies a strong law of large numbers with the same characteristic speed.  This paper is concerned with finer properties of the dynamics of the leftmost particle.

For $k \in \N$, we denote by~$a_k$ and $b_k$ the (attempted) jump rate to the left,  respectively, right of the $k$th particle. Throughout this paper, we assume that all the 
rates are strictly positive, and bounded from 
infinity by universal constants, 
i.e., that the following hypothesis,
 the intersection of Conditions~($\mathrm{A}_0$)
and~($\mathrm{A}_2$) of~\cite{MPW25}, is satisfied:

\begin{description}
\item\namedlabel{ass:rates}{$\text{A}$} 
There exists~$B\in(0,\infty)$
such that $0 < a_k \leq B$ and 
$0 < b_k \leq B$ for all~$k \in \N$.
\end{description}

The configuration space of the system is 
\begin{equation}
\label{eq:configuration-space}
\bbX  := \big\{ (x_1, x_2, \ldots ) \in \Z^{\N} : x_1 < x_2 < \cdots  \big\},  
\end{equation}
and the state of the process $X := (X(t))_{t \geq 0}$ at time~$t$ is
$X(t)=(X_1(t),X_2(t),X_3(t),\ldots)\in \bbX$
(that is, at time~$t$, $X_k(t)$ is the position of the $k$th particle). We will always assume that $X_1(0)=0$, which is no loss of generality for our questions of interest. 
Also define
\begin{equation}
\label{eq:eta-def}
\eta_k := (\eta_k (t))_{t \geq 0}, \text{ where } \eta_k (t) := X_{k+1}(t)-X_k(t)-1, 
\text{ for } k \in \N,
\end{equation}
  the number of  unoccupied sites 
between particles $k$ and $k+1$ at time~$t \in \RP$. 
The state of the system 
can thus be described by the position~$X_1$ of the leftmost particle  and by the process of inter-particle distances
$\eta := ( \eta_1,\eta_2,\ldots )$. 
Existence of a Markov process on $\bbX$ satisfying this informal description,
under the bounded rates hypothesis of
 Condition~\eqref{ass:rates}, is given by Proposition~1.6 of~\cite{MPW25}
  via a usual Harris graphical construction (see~\cite{FF94,andjel82,bmrs2017,ferrari92,bfl}),
  at least for all the initial conditions for $\eta(0)$ that we consider in~\cite{MPW25} and in this paper.

The process $\eta$ also has an interpretation
as an \emph{infinite Jackson network} of queues, as we describe in \S\ref{sec:intro-customer} below,
and this interpretation provides an associated \emph{customer random walk}
whose asymptotic behaviour is intimately linked with the dynamics of the particle
system, and the behaviour of the leftmost particle, in particular; one goal of this
paper is to explore further aspects of this connection, already partly investigated in~\cite{MPW25}.

A first step to understanding dynamics of the particle system is to investigate \emph{invariant measures} for the process $\eta$. Intuitively, if the system starts with $\eta$ in an invariant distribution, then the leftwards pressure felt by the leftmost particle due to the presence of the rest of the particle system is constant over large time-scales (since the fraction of time that the first inter-particle distance is equal to~$0$ is ergodic), which translates to a perturbation of the intrinsic speed of $X_1$ and hence a characteristic speed of the system as a whole. Since the configuration space $\bbX$ is uncountable, there may be many invariant measures, or none, depending on the $a_k, b_k$ (see below and~\cite{MPW25} for some examples). In the case of no invariant measures, the system is unstable, but one expects \emph{partial stability}, whereby the system can be decomposed into stable subsystems that barely interact. In the case of \emph{finite} systems, the partial stability picture was explored in~\cite{mmpw}. It turns out that central to describing these phenomena is investigating solutions $\rho := (\rho_k)_{k \in \ZP}$ to
the \emph{stable traffic equation} 
\begin{equation}
\label{eq:stable-traffic}
 (b_i +a_{i+1})\rho_i = a_i \rho_{i-1}+b_{i+1}\rho_{i+1}, \text{ for } i \in \N; ~ \rho_0 = 1,
\end{equation}
which is a 
linear system 
whose coefficients are the jump rate parameters 
$(a_i, b_i )_{i\in\N}$. 
(Throughout the paper, $\ZP := \{0\} \cup \N$.)

As discussed in~\S3 of~\cite{MPW25}, solutions to~\eqref{eq:stable-traffic}
form a one-parameter family $\rho = \rho(v) := \alpha + v \beta$, $v \in \R$,
where $\alpha := (\alpha_k)_{k \in \ZP}$
and $\beta := (\beta_k)_{k \in \ZP}$ are defined by 
\begin{align}
\label{eq:alpha-k-def}
 \alpha_0 & := 1, ~\text{and}~ \alpha_k := \frac{a_1\cdots a_k}{b_1\cdots b_k}  \text{ for } k \in \N; \\
\label{eq:beta-k-def}
\beta_0 & := 0 ,  ~\text{and}~ \beta_k := \frac{1}{b_k}+\frac{a_k}{b_k b_{k-1}}+
  \cdots + \frac{a_k\cdots a_2}{b_k\cdots b_1}
  \text{ for }k\in \N.
\end{align}
Solutions $\rho$ to~\eqref{eq:stable-traffic} are \emph{admissible}
if $\rho_k \in (0,1)$ for all $k \in\N$, and each admissible solution\footnote{
There may be many admissible solutions, unlike in the case of finite systems, where
there is another boundary condition which makes the solution of~\eqref{eq:stable-traffic} unique: see~\cite{mmpw} for the finite case.}
corresponds to a product-geometric stationary distribution $\nu_\rho$ (we give a precise
definition of~$\nu_\rho$ in terms of $\rho$ in~\S\ref{sec:upper-bound} below). 
If the process is started from a configuration with $\eta (0) \sim\nu_\rho$ 
for an admissible~$\rho = \rho(v)$, then $v$ is the stationary \emph{speed}  
of the process (i.e., for each fixed~$k$, $X_k(t)/t \to v$, a.s.; see Proposition 1.6 of~\cite{MPW25}). Among admissible solutions (if there are any),
distinguished is the \emph{minimal} solution $\rho(v_0) = \alpha + v_0 \beta$ where 
(see Definition~1.8 and Proposition~1.7 of~\cite{MPW25}) 
\begin{equation}
\label{eq:v0-def}
v_0 := - \lim_{k\to\infty} \frac{\alpha_k}{\beta_k}
 = -\Big(\frac{1}{a_1}+\frac{b_1}{a_1a_2}
  + \frac{b_1b_2}{a_1a_2a_3}+ \cdots\Big)^{-1} \in (-a_1,0].
\end{equation}
Then (i) $v_0 \leq v$ for every $v$ for which $\rho(v)$ is admissible, and (ii) $\nu_{\rho (v_0)}$ maximizes the probability, among all $\nu_{\rho(v)}$ for which $\rho(v)$ is admissible, of any particular finite collection of inter-particle distances all being~$0$. Roughly speaking, the minimal solution corresponds to the most densely packed stable configuration, hence the one with the greatest leftward pressure on the leftmost particle, and hence the most negative characteristic speed. It is also true that there are situations when no 
admissible solutions exist; as mentioned in Remark~1.5 
of~\cite{MPW25}, this usually means that the system 
can be decomposed into several ``stable subclouds'' which
do not interact with each other after some (random)
time. In any case, in this paper we will usually assume 
that at least one admissible solution \emph{does} exist (see Remarks~\ref{rems:known-finite}\ref{rems:known-finite-i} below for one way to verify that).

In the present paper, we will mainly
(but not in \S\ref{sec:stationarity-start} below) assume that 
the system starts from 
a configuration that is ``approximately close-packed''. Define
\begin{equation}
\label{eq:XB-def}
\bbXB := \big\{ x \in \bbX : x_{k+1} - x_k = 1 \text{ for all but finitely many } k \in \N \big\};
\end{equation}
we refer to $\bbXB$ as the set of \emph{finite configurations}, because there are only finitely many empty sites 
between particles.
An important observation is that, if the process starts from a finite initial 
configuration, then, almost surely, it will be still 
in~$\bbXB$ at any time~$t>0$.
The following summarizes the basic results of~\cite{MPW25} in that case.

\begin{proposition}
\label{prop:known-finite} 
    Suppose that Condition~\eqref{ass:rates} holds, and that there exists at least one admissible solution $\rho$ to~\eqref{eq:stable-traffic}. Take  $X(0) \in \bbXB$ with $X_1 (0) =0$. 
    Then, with $v_0$ given by~\eqref{eq:v0-def}, 
\begin{equation}
\label{eq:v_0_speed}
 \lim_{t\to\infty}\frac{X_k(t)}{t}=v_0, \text{ for every $k \in \N$}, \as
\end{equation}
For $\rho = \rho(v_0)$ the minimal solution to~\eqref{eq:stable-traffic}, we have that, for every finite $A \subset \N$,
\begin{equation}
\label{eq:local_convergence}
\lim_{t \to \infty} \IP \Big[\bigcap_{k \in A} \{ \eta_k (t) = u_k \} \Big] = \prod_{k \in A} (1-\rho_k) \rho_k^{u_k},
\text{ for every } u_k \in \ZP,\, k \in A.
\end{equation}
    Moreover, if $v_0 <0$ then the minimal solution is the only admissible solution to~\eqref{eq:stable-traffic}.
\end{proposition}

 We give a short proof of
Proposition~\ref{prop:known-finite} at the end of this introduction (\S\ref{sec:intro-customer}), indicating
how it is extracted from results of~\cite{MPW25}.

\begin{remarks}
\label{rems:known-finite} 
     \begin{myenumi}[label=(\roman*)]
     \item
     \label{rems:known-finite-i} 
     It is not hard to show (see (3.8) of~\cite{MPW25})
     that $\alpha_k/\beta_k$ is strictly decreasing in~$k$, and $0 \leq |v_0| < \alpha_k / \beta_k$,
     so that $\alpha_k +v_0 \beta_k > 0$ for all $k \in \ZP$. Since there exists an admissible $\rho$ if and only if $\rho(v_0)$ is admissible (see Proposition 1.7 of~\cite{MPW25}), and $v_0 \leq 0$, this means that to check that there exists an admissible $\rho$ it suffices to check that $\alpha_k < 1$ for all $k \in \N$, which turns out be convenient to check for our examples in this paper.
   \item
     \label{rems:known-finite-ii}
      The intuition for Proposition~\ref{prop:known-finite} is that, started from configurations with $X(0) \in \bbXB$ 
(which, by definition, are densely packed), among any admissible $\rho(v)$,
only $\rho (v_0)$ is ``accessible'' because, in light of the final statement
in Proposition~\ref{prop:known-finite}, any
non-minimal solution~$\rho(v)$ must have $v > v_0 = 0$, and positive speed
is impossible to achieve 
due to blocking by the eventually tightly-packed configuration to the right. The intuition
behind the fact that when $v_0 < 0$ there is only one admissible solution is that in this case
stability is manifest in the neighbourhood of the leftmost particle with particles travelling to the left, so the initial density of particles to the right is unable to influence the limit, although it will impact the speed of convergence. Indeed, we expect that when $v_0 <0$ the conclusion of
Proposition~\ref{prop:known-finite} remains valid for \emph{any} initial configuration, although this is not proved in~\cite{MPW25}. Finally, we note that much of this intuition can be expected to fail for general \emph{unbounded} rates where ``explosion'' phenomena appear possible; to our knowledge, that setting is  largely unexplored.
\end{myenumi}
\end{remarks}

One of the aims of the present paper is to study more closely the behaviour of the \emph{leftmost particle}, on a finer scale than the law of large numbers given in~\eqref{eq:v_0_speed}. While we do not have a complete picture, as we discuss in more detail below, we show that a rich variety of behaviours are possible, and we develop tools for classifying those behaviours. In this introduction, we state one result that shows the richness of the picture for a class of
rates $a_k,b_k$ that are asymptotically small perturbations of the homogeneous symmetric case where $a_k \equiv b_k$ are constant. By analogy with the classical near-critical phenomena for one-dimensional random walks explored in~\cite{lamp1,lamp2}, we call this class of parameters \emph{Lamperti-type} rates.

\begin{theorem}
    \label{thm:lamperti}
    Suppose that $0 < \mu < 1/2$, and take 
    \begin{equation}
        \label{eq:a-b-lamperti}
        a_k = \frac{1}{2} - \frac{\mu}{k} , ~ b_k = \frac{1}{2} + \frac{\mu}{k} , \text{ for all } k \in \N;
    \end{equation}
    in particular, Condition~\eqref{ass:rates} holds. 
          Take  $X(0) \in \bbXB$ with $X_1 (0) =0$. 
    Then, in addition to Proposition~\ref{prop:known-finite}, the leftmost particle process $X_1$ has the following behaviour.
 \begin{thmenumi}[label=(\alph*)]
    \item
        \label{thm:lamperti-a}
    If $0 <  \mu  < 1/4$, then $X_1$ is transient to $-\infty$ at polynomial rate, specifically,
    there exist constants $c_1, c_2$ (depending on $\mu$) with $0 < c_1 < c_2 < \infty$ and, a.s.,
    for all $t$ sufficiently large, 
    \begin{equation}
        \label{eq:poly-transient} 
 c_1 t^{\frac{1}{2} - 2\mu} < -  X_1 (t) < c_2  (t \log t)^{\frac{1}{2} - 2\mu} .
    \end{equation}
      \item 
             \label{thm:lamperti-b}
     If $ \mu > 1/4$, then   $X = (X_1, \eta)$ is positive recurrent, so, in particular, $X_1$ is ergodic.
    \end{thmenumi}
\end{theorem}
\begin{remarks}
\label{rems:lamperti}
     \begin{myenumi}[label=(\roman*)]
          \item
     \label{rems:lamperti-i}
     The critical case $\mu = 1/4$ is not covered in Theorem~\ref{thm:lamperti}, presents some subtleties, and  is unresolved: see Remark~\ref{rem:lamperti-critical} below, after the proof of Theorem~\ref{thm:lamperti}.
     \item
     \label{rems:lamperti-ii}
     The asymptotic speed~$v_0$ in~\eqref{eq:v_0_speed}, given by formula~\eqref{eq:v0-def},  satisfies $v_0=0$ in both parts~\ref{thm:lamperti-a} and~\ref{thm:lamperti-b}, where the stated results give finer information. 
        \item
     \label{rems:lamperti-iii}
     In part~\ref{thm:lamperti-b}, to say $X_1$ is ergodic means that, for every $A \subseteq \Z$,
     \begin{equation}
         \label{eq:X1-ergodic}
     \lim_{t \to \infty}    \frac{1}{t} \int_0^t \1 { X_1 (s) \in A } \ud s = \pi (A) , \as, 
     \end{equation}
     where the probability measure $\pi$ on $\Z$ is expressed in terms of $\nu_\alpha$ and $\eta(0)$ at~\eqref{eq:pi-def} below.
\item 
     \label{rems:lamperti-iv}
Since $0 < \mu < 1/4$ in part~\ref{thm:lamperti-a}, the exponent $\frac{1}{2}-2\mu$ in~\eqref{eq:poly-transient} can take any value in $(0,\frac{1}{2})$, i.e., the transience is \emph{subdiffusive}; this contrasts with the symmetric case where $\mu=0$, see remark~\ref{rems:lamperti-viii} below. It would be of interest to find examples where transience is polynomially \emph{superdiffusive} but sub-ballistic, i.e., exponent in $(\frac{1}{2},1)$ (see Problem~\ref{problem:superdiffusive} below). It is tempting to speculate that taking $-1/4 < \mu <0$ would achieve this, but that turns out to be false, transience being ballistic in that case: see remark~\ref{rems:lamperti-vii} below.
\item 
\label{rems:lamperti-v}
 The transience of $X_1$ in part~\ref{thm:lamperti-a}, started from $\eta(0) \in \bbXB$,
 should be contrasted with the fact that, for precisely the same rates, if we start from $\eta(0)$ drawn from the stationary $\nu_\rho$ corresponding to the minimal $\rho = \rho (v_0 ) = \alpha$,
 then $X_1$ is \emph{recurrent}, as shown in Theorem~\ref{thm:rec_stationary} below (see Remark~\ref{rem:importance-of-initial-condition}). Such examples show that the dynamics can depend crucially on the initial configuration, even in cases where there is a unique invariant measure.
\item    
\label{rems:lamperti-vi}
One of the main themes of this paper is to show a deeper interplay between the behaviour of the leftmost particle in the semi-infinite particle system and the characteristics of a much simpler stochastic process, a certain random walk on $\ZP$, called the \emph{customer random walk}, that we introduce in~\S\ref{sec:intro-customer} below. To preview this aspect, we say here that part~\ref{thm:lamperti-b} of Theorem~\ref{thm:lamperti} corresponds to the case when the customer walk is  \emph{positive recurrent}, and part~\ref{thm:lamperti-a} to when it is \emph{null-recurrent}. Previously, it was known that
$v_0 <0$ if and only if the customer walk is transient (see Remark 1.11 of~\cite{MPW25}). 
\item 
\label{rems:lamperti-vii}
The model with rates~\eqref{eq:a-b-lamperti} is well defined for all $| \mu | < 1/2$,
but it is only in the case $0 < \mu < 1/2$ that there is an admissible solution to~\eqref{eq:stable-traffic}, meaning that 
we are in the setting of Proposition~\ref{prop:known-finite}, which is
the starting point for this paper. 
In the case $-1/2 < \mu < 0$, 
all particles are singleton clouds and travel to the left at their own intrinsic speeds (in excess of $|v_0|$), so the collective behaviour that we are interested in here is absent. It is also worth noting that the restriction $\mu < 1/2$ is needed so that $a_1 >0$ in~\eqref{eq:a-b-lamperti}, but part~\ref{thm:lamperti-b} would still hold true if~\eqref{eq:a-b-lamperti}
was assumed for all $k \geq k_0$ large enough, provided $\rho(0)=\alpha$ is assumed admissible; this would mean all $\mu >1/4$ could be included.
\item 
\label{rems:lamperti-viii}
In the case $\mu =0$, the model of~\eqref{eq:a-b-lamperti} is the \emph{symmetric simple exclusion process} 
and 
 a result of Arratia, Theorem~2 of~\cite[p.~368]{arratia83},
 says that 
 \begin{equation}
     \label{eq:arratia}
\lim_{t\to\infty}  \frac{X_1 (t)}{\sqrt{ t \log t}} = 1, \as
  \end{equation}
As mentioned in the previous remark, the case $\mu =0$ has no admissible solution, 
but nevertheless seems a reasonable comparison for our bounds in~\eqref{eq:poly-transient}; this comparison suggests that it is perhaps the upper bound in~\eqref{eq:poly-transient} that is of the correct order.
    \end{myenumi}
\end{remarks}

As mentioned in Remarks~\ref{rems:lamperti}\ref{rems:lamperti-iv}, Theorem~\ref{thm:lamperti} exposes a natural question:

\begin{problem}
    \label{problem:superdiffusive}
    Do there exist rates $a_k, b_k$ satisfying Condition~\eqref{ass:rates} for which, when started from a finite initial configuration,  the leftmost particle is transient with a polynomially superdiffusive but sub-ballistic rate, i.e., 
    $\log |X_1 (t) | / \log t \to \gamma$ for some $\gamma \in (1/2,1)$?
\end{problem}

Before describing in more detail the customer random walk (in the next section), 
we indicate the other
main contributions of this paper, in addition to
Theorem~\ref{thm:lamperti} above.
\begin{itemize}
    \item
Theorem~\ref{thm:lamperti} shows examples where $X_1$ is positive recurrent, and where it is transient. Examples where $X_1$ is \emph{null} recurrent (appropriately defined) seem to be rarer. We present one such example in Theorem~\ref{thm:null-recurrent} below.
\item 
We also consider the dynamics of $X_1$ started from \emph{stationary} configurations. In Theorem~\ref{thm:rec_stationary} below
we show that in that case $X_1(t) - v t$ is recurrent when $\eta(0) \sim \nu_\rho$
for any admissible $\rho = \rho(v)$. In particular, either (i) $v <0$ (which can only be $v = v_0 <0$) in which case $X_1$ is ballistically transient to the left but oscillates on both sides of its strong law,  (ii) $v=0$ so $X_1$ is recurrent, or (iii) $v > v_0 =0$ so $X_1$ is ballistically transient to the right. In particular, the sub-ballistic transience of Theorem~\ref{thm:lamperti}\ref{thm:lamperti-a} is shown to be possible only when starting \emph{away from stationarity}.
\end{itemize}

\subsection{Introducing the customer random walk}
 \label{sec:intro-customer}

It is well known (see e.g.~\cite{kipnis})
that the inter-particle distances in nearest-neighbour exclusion processes on $\Z$
correspond exactly to Jackson networks of queues; see \S2.1 of~\cite{MPW25}
for a discussion in precisely our setting,   \S3 of~\cite{mmpw} for   the case of finitely many particles, and references therein for more background. 
The queueing representation considers the empty sites between consecutive
particles as \emph{customers} in the (in the present case, infinite) queueing network; we interpret the number of empty sites
between particles $k$ and $k+1$ as the number of customers at queue~$k \in \N$. Jumps in the particle system correspond to customers in the queueing system being served at one queue and then routed to a neighbouring queue, or, specially, 
jumps of the leftmost particle bring customers into the system (if it jumps left) or eject customers from the system (if it jumps right). 
For example, when the second particle jumps to the left
(thus reducing the number of empty sites between the first and the second particles by~$1$, and increasing the number of empty sites between the second and the third particles by the same amount), we may interpret it as ``a customer from the first queue was served, and then went to the second queue''. Notice, in particular,
that a particle's jump to the left implies a customer's jump to the right, and vice-versa.

We state here the formal definition of  the \emph{customer random walk}; one of the main themes of this paper is to explore its connections with the process~$X_1$, i.e.,
the movement of the leftmost particle.

\begin{definition}[Customer random walk]
\label{def:customer_walk}
The customer random walk is  a continuous-time nearest-neighbour random walk~$\zeta := (\zeta_t)_{t \in \RP}$ 
with state space $\ZP$, where
transitions from $k \in \ZP$ to $k+1$ occur at rate $a_{k+1}$
and transitions from $k \in \N$ to $k-1$ occur at rate $b_k$. Unless explicitly specified otherwise,
we will always assume in our calculations 
that~$\zeta_0 = 1$.
\end{definition}
\begin{remark}
    \label{rem:customer_walk}
    The random walk $\zeta$ is essentially  the continuous-time version 
of the walk~$Q$ from \S2.3 of~\cite{MPW25},
and corresponds to the progress of a ``priority customer''
through the queueing network (the priority customer is always served ahead of any other customer in the same queue). In fact, from the point of view of
such a customer, $0$ would be an absorbing state (because entry to state~$0$ represents the customer 
leaving the system), but we make the walk~$\zeta$ irreducible (under Condition~\eqref{ass:rates}) by 
assigning rate~$a_1$ to transitions from~$0$ to~$1$.
\end{remark}

We use $\IP$ and $\IE$ for probability and expectation statements involving $\zeta$,
although there is no need to imagine that $\zeta$ is defined on the same probability space as our particle system~$X$.

We collect some important observations about the random walk~$\zeta$ under
Condition~\eqref{ass:rates}, so that $\zeta$ is irreducible. 
Let~$\tau := \inf \{ t \geq 0 : \zeta_t = 0\}$ be the hitting time of~$0$ for the 
customer random walk. 
It is straightforward to check that~$\alpha$ given by~\eqref{eq:alpha-k-def} is a
reversible measure for~$\zeta$, meaning that~$\zeta$
is positive recurrent (i.e., $\IE \tau < \infty$) if and only if $\sum_{k \in \ZP} \alpha_k<\infty$. On the other hand,
it is a standard result that $\zeta$ is transient ($\IP [\tau =\infty] >0$) if and only if $\sum_{k \in \ZP} (a_{k+1} \alpha_k)^{-1} <\infty$. Hence $\zeta$ is null recurrent if
$\sum_{k \in \ZP} \alpha_k = \sum_{k \in \ZP} (a_{k+1} \alpha_k)^{-1} = \infty$. (See e.g.~\cite[Ch.~8]{anderson} for these well known facts about birth-death processes.)

See Proposition~\ref{prop:walk-summary-known} below for some basic results about how transience or positive recurrence of $\zeta$ leads directly to statements about the dynamics of the leftmost particle process~$X_1$. As we will see in \S\ref{sec:finite_initial}, when $\IP [ \tau < \infty ] =1$, 
finer information about the distribution of~$\tau$ plays an important role in understanding the finer asymptotics of $X_1$. 

The rest of the paper is organized as follows. In \S\ref{sec:finite_initial}, we investigate the behaviour
of the leftmost particle in the case of finite initial configurations. In particular, we give upper and lower bounds on the asymptotic location $X_1(t)$, using a comparison with the M/G/$\infty$ queue and recent results for that~\cite{P25}, and comparison with the process started from stationarity. These tools allow us to establish Theorem~\ref{thm:lamperti} above, and
by a similar method we prove
Theorem~\ref{thm:null-recurrent}, exhibiting an example where the leftmost particle is null recurrent. 
Then in \S\ref{sec:stationarity-start} we consider in more detail the case of stationary initial distributions; the main result here is Theorem~\ref{thm:rec_stationary} which shows oscillation around the stationary strong law behaviour. 

Before proceeding with the main body of the paper, we clarify a minor error from~\cite{MPW25} and then
record the proof of Proposition~\ref{prop:known-finite}.

\begin{remark}
    We take the opportunity here to point out and correct a small but misleading error in~\cite{MPW25}.
The error originates in part (vi) of Proposition 3.1 of~\cite{MPW25},
which should say that if an admissible $\rho = \rho(v)$ is such that $\liminf_{k \to \infty} \rho_k =0$, then $v=0$ (the statement incorrectly claims $v_0=0$ and $\cV = \{0\}$, but the proof only gives $v=0$ for the particular~$v$ in question).     
    This does not impact the main results of~\cite{MPW25}, but does mean that the formulation of Remark~1.15 is incorrect. The correct remark is that under the bounded rates hypothesis,
    Proposition 3.1(iv) of~\cite{MPW25} shows that non-uniqueness of admissible solutions can occur only when $v_0=0$, so any non-minimal admissible solutions $\rho (v)$ must have $v>0$. 
    Other minor changes to the reading of~\cite{MPW25} required to remove reference to the incorrect claim in Proposition 3.1(vi)
    are to delete the
    sentence after the proof of Lemma 4.2, and, just below (2.8) to replace ``$v=v_0=0$'' by
    ``$v=0$'' (which is all that is needed there). 
\end{remark}

\begin{proof}[Proof of Proposition~\ref{prop:known-finite}]
First, the fact that $v_0 < 0$ implies there is a unique admissible solution,
under Condition~\eqref{ass:rates}, is given in Proposition~3.1(iv) of~\cite{MPW25}. Convergence to stationarity is Theorem 1.14 of~\cite{MPW25}. The strong law of large numbers~\eqref{eq:v_0_speed} is a consequence of Theorem 1.9 of~\cite{MPW25}. 
Here, we note that, in order to be able to apply that theorem 
in the case $v_0<0$, we need to check that 
$\overline\beta := \limsup_{k\to\infty}\beta_k = \infty$;
let us show that this is indeed true under 
Condition~\eqref{ass:rates}. First, 
note that we can assume that 
$c:=\liminf_{k\to\infty}\alpha_k>0$
(otherwise, we would obtain a contradiction with the fact
that $\beta_k\geq B^{-1}$ for all~$k$ but 
$\rho=\alpha-|v_0|\beta$ is admissible).
Then, the generic term in~\eqref{eq:beta-k-def}
can be bounded from below as follows:
\begin{equation}
\label{bounding_term_beta_k}
 \frac{a_k\cdots a_{k-m+1}}{b_k\cdots b_{k-m}}
  = \frac{\alpha_k}{b_{k-m}\alpha_{k-m}}
  \geq B^{-1}\frac{\alpha_k}{\alpha_{k-m}}.
\end{equation}
Now, if $\overline\alpha := \limsup_{k\to\infty}\alpha_k < \infty$,
then the right-hand side of~\eqref{bounding_term_beta_k}
is at least $B^{-1}c\overline\alpha^{-1}$, so the sequence
$(\beta_k)_{k\geq 1}$ must grow to infinity (even linearly).
On the other hand, if $\overline\alpha=\infty$, then 
the sequence $(\alpha_k)_{k\geq 1}$ contains a strictly 
increasing subsequence $(\alpha_{k_n})_{n\geq 1}$ 
converging to infinity and such that
$\alpha_{k_n}=\max_{\ell\leq k_n}\alpha_\ell$ for all~$n$.
But then the right-hand side of~\eqref{bounding_term_beta_k}
(with~$k_n$ on the place of~$k$) is at least~$B^{-1}$,
meaning that $\beta_{k_n}\geq B^{-1}k_n$,
which again shows that $\overline\beta=\infty$.
\end{proof}

\section{Finite initial configurations}
\label{sec:finite_initial}

\subsection{Overview}
\label{sec:finite_initial_overview}

The goal of this section is to investigate how 
the recurrence/transience properties of the customer random walk
influence the dynamics of the leftmost particle.
Recall the definition of the queueing process $\eta = (\eta (t))_{t \geq 0}$ from~\eqref{eq:eta-def}, with configurations
$\eta(t) \in \bbD := \ZP^{\N}$.
For a generic $u = (u_k)_{k \in \N} \in \bbD$, we write
\begin{equation}
    \label{eq:u-norm}
\| u \| := \sum_{k \in \N} u_k.
\end{equation}
Note also that, if $\| \eta (0) \| < \infty$, then $\| \eta(t) \| < \infty$ for all $t \geq 0$, and since every customer to enter/leave the queueing system represented by $\eta$ means that the leftmost particle steps to the left/right, we have
\begin{equation}
    \label{eq:queue-leftmost-finite}
    X_1 (0) - X_1 (t) = \| \eta (t) \| - \| \eta (0) \|, \text{ whenever } X(0) \in \bbXB.
\end{equation}

As noted in Remark~1.11 of~\cite{MPW25},
$|v_0|/a_1 = \IP [ \tau = \infty]$ is the escape probability
of the customer random walk~$\zeta$ (see \S\ref{sec:intro-customer} for definitions), i.e., it is the probability
of never reaching~$0$ starting at~$1$.
Since, when $X(0) \in \bbXB$,
$|X_1(t)| = | X_1(t) - X_1(0)|$ is equal to the net inflow of customers to 
the system by time~$t$, as in~\eqref{eq:queue-leftmost-finite}, 
and customers enter the system at rate $a_1$, 
\begin{equation}
\label{eq:ballistic_left}
\liminf_{t\to\infty}\frac{|X_1(t)|}{t}
\geq a_1 \IP [ \tau=\infty ]=|v_0|, \as, \text{ whenever } X(0) \in \bbXB; 
\end{equation}
i.e., if the customer random walk is transient, then the leftmost
particle goes to the left ballistically.
Note that for~\eqref{eq:ballistic_left}
we do \emph{not} need to assume existence 
of admissible solutions:
in that case,
the strong law in Proposition~\ref{prop:known-finite} (which follows from 
Theorem~5.1(ii)
of~\cite{MPW25}) says that the $\liminf$ in~\eqref{eq:ballistic_left}
is a limit, and the inequality an equality. In the more general setting, 
when no admissible solutions exist, 
the inequality 
in~\eqref{eq:ballistic_left} may be strict,
when a finite stable subcloud detaches 
from the main system and goes to $-\infty$ 
with a speed strictly greater than~$|v_0|$. For example, it can be that the customer walk is recurrent (a property determined by $a_k, b_k$ for large $k$ only) so that the right-hand side of~\eqref{eq:ballistic_left} is equal to~$0$, but nevertheless the first few particles separate from the bulk (at a speed determined by $a_k, b_k$ for small $k$ only) and hence the total number of customers in the system diverges.

It holds (see Lemma~4.1 of~\cite{MPW25})
that if $\sum_{k \in \ZP} \rho_k < \infty$, 
then the measure~$\nu_\rho$ given by~\eqref{eq:df_measure_rho} is supported
on configurations~$u \in \bbD$ with $\| u \|< \infty$ (recall~\eqref{eq:u-norm}). That is, if $\sum_{k \in \ZP} \rho_k < \infty$, 
then $\nu_\rho ( \bbDB ) = 1$ where
\begin{equation}
\label{eq:D-F-def}
\bbDB := \Bigl\{u \in \bbD : \| u \| < \infty \Bigr\}.  
\end{equation}
(Note that $\eta (t) \in \bbDB$ if and only if $X(t) \in \bbXB$.)
Also, Proposition~3.1(v) 
and Lemma~4.2 of~\cite{MPW25} show that
if there is an admissible solution~$\rho$ with $\sum_{k \in \ZP} \rho_k < \infty$,
then in fact $v_0 = 0$ and  $\sum_{k \in \ZP} \alpha_k < \infty$,
and $\nu_\alpha$ is the unique invariant measure supported 
on $\bbDB$. 
Moreover, Theorem~1.12 of~\cite{MPW25}
shows that if $\sum_{k \in \ZP}\alpha_k<\infty$ and there exists an
admissible solution, then
the queue process $\eta$ is positive recurrent
whenever we start from a   configuration $\eta (0)$ with finitely many initial customers in the system 
(in this case, $\eta$ is a countable Markov chain living on the space~$\bbDB$ defined in~\eqref{eq:D-F-def}). 
Then, by~\eqref{eq:queue-leftmost-finite}, positive recurrence of $\eta$ implies positive
recurrence of $X = (X_1,\eta)$, and that $X_1$ is ergodic in the sense of~\eqref{eq:X1-ergodic},
where
     \begin{equation}
         \label{eq:pi-def}
     \pi (x) := \nu_\alpha \bigl\{ u \in \bbD : \| u \|  =  \|  \eta (0) \| - x \bigr\} , \text{ for } x \in \Z.
     \end{equation}
We often will simply say ``$X_1$ is positive recurrent'' in this case.



Recalling that positive recurrence of~$\zeta$ is equivalent to $\sum_{k \in \ZP} \alpha_k<\infty$,
as discussed in~\S\ref{sec:intro-customer}, the previous discussion (which mostly  recalls results from~\cite{MPW25}) can thus be summarized as follows.

\begin{proposition}
\label{prop:walk-summary-known}
    Suppose that Condition~\eqref{ass:rates} holds and that $X(0) \in \bbXB$ is a finite initial configuration. Recall that $\zeta$ denotes the customer random walk from Definition~\ref{def:customer_walk}.
 \begin{thmenumi}[label=(\alph*)]
        \item 
        \label{prop:walk-summary-known-a}
If $\zeta$ is transient, then $X_1$ is ballistically transient to $-\infty$ with speed at least $|v_0| >0$, as given by~\eqref{eq:ballistic_left}.
\item 
\label{prop:walk-summary-known-b}
If there exists an admissible solution with $\nu_\rho (\mathbb{D}_F) =1$, and  $\zeta$ is positive recurrent, then
$X_1$ is positive recurrent, satisfying~\eqref{eq:X1-ergodic}.
    \end{thmenumi}
\end{proposition}
\begin{remark}
    \label{rem:walk-summary-known}
    For Proposition~\ref{prop:walk-summary-known}\ref{prop:walk-summary-known-b}
    it is essential to assume existence of admissible
solutions: otherwise, since modifying~$a_1$
does not affect the properties of~$\zeta$,
it is straightforward to construct an example where the leftmost
particle goes to $-\infty$ ballistically even though~$\zeta$
is positive recurrent.
\end{remark}

For the results that we establish later in this section, we  need to recall 
some more detailed information about the stationary measures $\nu_\rho$
corresponding to admissible solutions $\rho$ to the stable traffic
equation~\eqref{eq:stable-traffic} described in \S\ref{sec:intro-dynamics}.
Denote by $\geo{q}$ the (shifted) geometric distribution on $\ZP$ with success parameter~$q \in (0,1]$, i.e., $\xi \sim \geo{q}$ means that $\IP[\xi = n] = (1-q)^n q$ for $n \in \ZP$;
note then $\IE \xi = (1-q)/q$. For an admissible solution~$\rho$, 
let $\nu_\rho$ be the product measure $\bigotimes_{k \in \N} \geo{ 1-\rho_k}$ on $\bbD = \ZP^\N$, i.e.,
\begin{equation}
\label{eq:df_measure_rho}
\nu_\rho ( u )
   = \prod_{k\in A} (1-\rho_k) \rho_k^{u_k}, 
   \text{ for all  finite } A\subset \N 
   \text{ and all } u = (u_k)_{k \in A}  \in \ZP^A.
\end{equation}
In Proposition~1.6 of~\cite{MPW25} it was shown that~$\nu_\rho$ is an invariant measure for the queue process~$\eta$ (that is, if we start the queue process by choosing the initial configuration according to~$\nu_\rho$, which we write $\eta(0) \sim \nu_\rho$, 
then $\eta (t) \sim \nu_\rho$ at any~$t>0$).

Proposition~\ref{prop:walk-summary-known} shows
the implications of transience or positive recurrence of the customer random walk~$\zeta$
for the dynamics of the leftmost particle $X_1$. For the remainder of this section,
we will investigate the case where the customer random walk
is \emph{null-recurrent} (so that $\sum_{k \in \ZP} \alpha_k = \infty$)
and an admissible solution exists. 
An example of such a situation
was treated in Theorem~1.19 of~\cite{MPW25},
where the customer random walk was essentially 
a simple symmetric random walk;
see also~\cite{arratia83}.
In general, from the previous discussion we can say that
in this situation $v_0=0$, and hence the motion of the leftmost
particle is not ballistic, by~\eqref{eq:v_0_speed}. On the other hand, $X_1$
is not positive recurrent (if it were so, then so would be 
the queue process~$\eta$, given that the fact that 
$X_1$ reached its rightmost possible position means 
that~$\eta$ reached the empty configuration).

This leaves the possibility that $X_1$ can be transient or null recurrent;
we show that both are possible, although null-recurrent examples seem rare (see
Theorem~\ref{thm:null-recurrent} in \S\ref{sec:null_recurrent} below).
In the transient case, we investigate the 
``sub-ballistic rate'' at which~$|X_1|$ converges to infinity. Here Theorem~\ref{thm:lamperti}, that we prove in \S\ref{sec:Lamperti} below, gives a class of examples with \emph{subdiffusive} transience at all possible polynomial rates in $(0,1/2)$ ($1/2$ corresponding to diffusivity). 
In the rest of \S\ref{sec:finite_initial}, we discuss
these questions. First in \S\S\ref{sec:MGinfinity}--\ref{sec:upper-bound} we provide
some quite general results that use more detailed information about the tail of~$\tau$
to obtain quantitative bounds on the growth rate of $|X_1|$.

Before continuing,
let us define another quantity that we need, called the 
\emph{scale function} 
for the customer random walk:
\begin{equation}
\label{eq:scale_func-def}
f(0) :=  0, \text{ and } 
f(k) := \frac{1}{a_1} +\frac{b_1}{a_1 a_2} +
\cdots  + \frac{b_1\cdots b_{k-1}}{a_1 a_2 \cdots a_k}, ~ k \in \N.
\end{equation} 
It is straightforward to check that the process 
$f(\zeta_{t\wedge \tau})$ is a martingale; this fact 
can be conveniently used to estimate hitting probabilities
for~$\zeta$ via the optional stopping theorem.

\subsection{Lower bound: comparison with the \texorpdfstring{$M/G/\infty$}{M/G/infinity} queue}
\label{sec:MGinfinity}

 An  $M/G/\infty$
queue is defined in the following way:
customers arrive according to a Poisson
process with rate~$\lambda$; upon arrival, a customer     
 enters to service, and the service times
are i.i.d.\ non-negative 
random variables with some general distribution;
 let~$S$ be a generic random variable 
with that distribution. 
Denote by~$Y_t$ the number of customers
in the system at time~$t$; 
we say that the system is transient
if $Y_t\to\infty$ a.s., and recurrent
if $\liminf_{t\to\infty}Y_t = 0$ a.s.\
(notice that this implies that the system visits
all its ``states'' infinitely many times a.s.).

The relevance of the  $M/G/\infty$
queue for our semi-infinite particle system is due to a stochastic domination property which says
that the negative displacement of the leftmost particle in the particle system
dominates an appropriate $M/G/\infty$ queue. Heuristically, this is due to the fact that a customer
in the particle system may get delayed
because (s)he has to compete for service
with other customers. 
In other words, if we consider a hypothetical situation when there is an infinite number
of servers in all queues so that all customers enter to service immediately without
any competition, then the total number of customers in our system would be precisely
described by an $M/G/\infty$ queue (the service time~$S$ of the $M/G/\infty$ queue would 
correspond to the total time that the customer spends in the system; note that, without competition, these would become customer-wise independent).
A precise statement is the following; we defer the proof to the end of this section.

\begin{lemma}
\label{lem:MGinfty-dominance}
Let $u \in \bbDB$. 
  There exists a probability space, with probability $\tIP$,
    supporting stochastic processes $X = (X_1,\eta)$ on $\bbX = \Z \times \bbDB$
    and $Y$ on $\ZP$, where $X$ has the law of the
    particle system under $\IP$ started from $X_1(0) =0$ and $\eta (0) = u$,
    and~$Y$ has the law of an $M/G/\infty$ queue with arrival rate~$a_1$, service 
time distributed as~$\tau$ (the hitting time of~$0$
for the customer random walk $\zeta$ started at~$1$, as defined in \S\ref{sec:intro-customer}),
and $Y_0 =0$, such that
\[ \tIP ( Y_t \leq \|u \| - X_1 (t) \text{ for all } t \geq 0 )  =1. \]
\end{lemma}

This domination result gives a strategy to obtain a lower bound on $|X_1|$ via a lower bound on the 
$M/G/\infty$ 
queue length process $Y$ with arrival rate $\lambda = a_1$. 
The latter was studied in~\cite{P25}, and  we reproduce here 
the key results from 
Theorems~1 and~2 of~\cite{P25}. First, 
 if 
\begin{equation}
\label{cond_MGinfty_rec}
 \int_0^\infty 
 \exp\big(-\lambda \IE(S\wedge t)\big)\, \ud t = \infty,
\end{equation}
then the $M/G/\infty$ system is recurrent.%
\footnote{In fact, the condition~\eqref{cond_MGinfty_rec} assures that 
the expected size of the set $\{t: Y_t=0\}$ is infinite; although the process~$Y$ is not Markovian,
it is still possible to argue that the preceding fact implies the recurrence.} 
On the other hand, if
\begin{equation}
\label{cond_MGinfty_trans}
 \int_0^\infty \big(\IE(S\wedge t)\big)^k
 \exp\big(-\lambda \IE(S\wedge t)\big)\, \ud t < \infty
,  \text{ for all } k \in \ZP, 
\end{equation}
then the $M/G/\infty$ system is transient.
Moreover, in the transient case,
for $q\in(0,1)$, define 
\begin{equation}
\label{eq:gamma-q-def}
\gamma_q := 1 - q + q\log q > 0.
\end{equation}
If it holds that, for some $q\in(0,1)$, 
\begin{equation}
\label{cond_growth_Y}
\int_0^\infty \exp\big(-\gamma_q\lambda \IE(S\wedge t)\big)\, \ud t < \infty,
\end{equation}
then\footnote{In fact, \eqref{cond_growth_Y} implies that the expected size of the 
set $\{t: Y_t < q\lambda \IE(S\wedge t)\}$ is finite, and then it is possible to argue that 
this set has to be bounded a.s.}
\begin{equation}
\label{eq_growth_Y}
\IP\big[Y_t\geq q\lambda \IE(S\wedge t)
\text{ for all large enough }t\big]=1.
\end{equation}

 Therefore, as a corollary
of the stochastic domination described above, and the
results from~\cite{P25} on transience just quoted, 
we obtain the following.

\begin{theorem}
\label{thm:MGinfty_comparison}
Suppose that Condition~\eqref{ass:rates} holds, and that $\eta(0) \in \bbDB$. 
Let~$\tau$ be the hitting time of~$0$ for the customer
random walk~$\zeta$ started at~$1$.
\begin{thmenumi}[label=(\alph*)]
        \item
        \label{thm:MGinfty_comparison-a}
        Suppose that
\begin{equation}
\label{cond_part_trans}
 \int_0^\infty \big(\IE(\tau\wedge t)\big)^k
 \exp\big(-a_1 \IE(\tau\wedge t)\big)\, \ud t
  < \infty, \text{ for all } k \in \ZP.
\end{equation}
Then $X_1(t)\to -\infty$ a.s., as $t\to\infty$.
\item 
\label{thm:MGinfty_comparison-b}
Suppose that, for some $q\in(0,1)$ and with $\gamma_q$ defined at~\eqref{eq:gamma-q-def},
\begin{equation}
\label{cond_growth_X_1}
\int_0^\infty \exp\big(-\gamma_q a_1
\IE(\tau\wedge t)\big)\, \ud t < \infty.
\end{equation}
Then 
\begin{equation}
\label{eq_growth_X_1}
\IP\big[X_1(t)\leq -q a_1 \IE(\tau\wedge t)
\text{ for all large enough }t\big]=1.
\end{equation}
\end{thmenumi}
\end{theorem}
\begin{remark} 
    \label{rem:MGinfty_comparison}
    Note that Theorem~\ref{thm:MGinfty_comparison} does not require the existence of an admissible solution.
\end{remark}

We will use this result in \S\ref{sec:Lamperti},
but, for now, let us make the following
observation.
There is a gap between~\eqref{cond_MGinfty_rec}
and~\eqref{cond_MGinfty_trans},
in the sense that it is possible to choose
the distribution of~$S$ in such a way that
neither of these two relations hold. 
This is because,
as shown in Theorem~1 of~\cite{P25},
for an $M/G/\infty$ queue it is possible to have 
coexistence of recurrent and transient states
(i.e., to have 
$\liminf_{t\to\infty} Y_t = \kappa$ a.s.~for a constant $\kappa \in (0, \infty)$).
It is then natural to ask whether the following  coexistence
phenomenon can occur in our particle system:

\begin{problem}
Do there exist rate parameters satisfying~Condition~\eqref{ass:rates}
for which it holds that $0 < \liminf_{t \to \infty} |X_1 (t) | < \infty$, a.s.?
\end{problem}

In such a situation,  there would be
always some empty sites in the particle system configuration, but their 
number does not converge to infinity.
For now, we do not have any further 
insights on this.

To finish this section, we will give the proof of Lemma~\ref{lem:MGinfty-dominance}.
Here (and in other stochastic comparison arguments later on) it is useful to use the concept of \emph{second-class customers} (as in \S4.2 of~\cite{MPW25}) to compare different initial conditions.
For fixed $u \in \bbD$, we write~$\IP_u$ to denote the law of the particle process with initial configuration $X_1 (0) =0$ and $\eta (0) = u$; similarly,
for~$\nu_\rho$ one of the stationary measures given by~\eqref{eq:df_measure_rho}, we write $\IP_{\nu_\rho}$ for initial condition $X_1 (0) =0$ and $\eta (0) \sim \nu_\rho$.
Sometimes we will be only concerned with the queueuing process $\eta$ (and not~$X_1$), in which case we may still refer to $\IP_u$ and $\IP_{\nu_\rho}$ to specify laws of~$\eta$ alone.

Suppose $\eta(0) = u \in \bbD$, and declare all customers in the queueing network at time~$0$ to be second class; any customers that arrive subsequently are first class. First-class customers get priority, so
whenever a service event occurs, if there is at least one first-class customer in the queue, it is a first-class customer that is served. Then observing all the customers in the system, we see a process following law $\IP_u$, while observing only the first-class customers we see a process following law $\IP_\0$ started from $\0 := (0,0,\ldots) \in \bbD$ (empty). This construction (and a similar one started from $\nu_\rho$) gives the following stochastic monotonicity properties for the queueing process (cf.~Proposition 4.3 of~\cite{MPW25}).

\begin{lemma}
    \label{lem:queue-monotonicity}
    For every $u \in \bbD$, we can build on a common probability space $(\Omega, \cF, \tIP)$ processes $\eta^\0$ and $\eta^u$
    for which $\eta^0$ has law $\IP_\0$, $\eta^u$ has law $\IP_u$, and $\tIP ( \eta^\0_k (t) \leq \eta^u_k (t) \text{ for all } k \in \N \text{ and all } t \geq 0 ) =1$. The same is true for $\eta^0$ and $\eta^{\nu_\rho}$, where $\eta^{\nu_\rho}$ has law $\IP_{\nu_\rho}$ for an admissible~$\rho$.
\end{lemma}

\begin{proof}[Proof of Lemma~\ref{lem:MGinfty-dominance}]
The result essentially follows from Proposition~2.3 of~\cite{MPW25}. 
Indeed, consider the queueing process~$\eta$ started with no initial customers. 
For definiteness, suppose first-in, first-out service. For each customer,
if we count only accumulated service time when the customer is ``at the front of the queue'',
their total time in the system is distributed as~$\zeta$ (see Remark~\ref{rem:customer_walk}); in the 
$M/G/\infty$ queue, all customers are always ``at the front of the queue''. Hence 
there is a coupling in which each customer enters both systems at the same time, but stays in the $\eta$ system for at least as long as they stay in the $M/G/\infty$ system, and hence 
$\tIP ( \| \eta (t)\| \geq Y_t \text{ for all } t \geq 0) = 1$ in that coupling.
But by~\eqref{eq:queue-leftmost-finite} (recall $\|  \eta (0)\|  = 0$ for now) $\| \eta (t)\| = - X_1(t)$, which establishes the claim in the lemma for the case $\| u \| =0$. 

In the general case, we start with $\eta (0) = u \in \bbDB$. Treat these $\| u \|$ initial customers as second-class customers in the $\eta$ system, and then couple the process of subsequent first-class customers in the $\eta$ system to the $M/G/\infty$ queue, as before, to see that 
$\tIP ( \| \eta (t)\| \geq Y_t \text{ for all } t \geq 0) = 1$, still with $Y_0 = 0$,
 but now  $\| \eta (t)\| - \| \eta (0)\| = - X_1(t)$ by~\eqref{eq:queue-leftmost-finite}.
\end{proof}

\subsection{Upper bound: using the stationary distribution}
\label{sec:upper-bound}

The comparison with the  $M/G/\infty$ queue from \S\ref{sec:MGinfinity} only seems useful to
obtain lower bounds on $|X_1|$, because of the direction of the stochastic comparison we discussed there.
For upper bound on $|X_1|$, we take a quite different approach.

We now state
a general result about
an upper bound on the growth of~$|X_1|$.

\begin{theorem}
\label{thm:upper_growth}
Suppose that Condition~\eqref{ass:rates} holds, that $v_0 =0$, and 
$\alpha$ is an admissible solution. 
Let~$h: \RP\to (1,\infty)$ be a continuous, 
 increasing to infinity, differentiable 
function such that $h'(t)\leq a_1$ for all~$t$.
Suppose also  that there exist functions
$\phi : \ZP \to [0,1]$ and $\ell: \ZP\to\RP$ with $\ell(0)=0$ and $\ell$ increasing to infinity, such that
\begin{equation}
\label{est_stat_measure_ell_k}
 \nu_\alpha \Bigl\{ u \in \bbD : \sum_{j=1}^k u_j>\ell(k) \Bigr\}
  \leq \phi(k), \text{ for every } k \in \N,
\end{equation}
and, with $\zeta$ denoting the customer random walk, 
\begin{equation}
\label{eq:upper_growth_cond}
 \int_1^{\infty} \Bigl( \phi(k_t)
  + t\IP\big[\max_{0 \leq s\leq t}\zeta_s
   \geq k_t+1\big] \Bigr)\,\ud t<\infty,
\end{equation}
where $k_t : = \max \{k \in \ZP : \ell(k) \leq h(t)-1 \}$.
Take an initial configuration  $\eta (0) \in \bbDB$. Then
\begin{equation}
\label{eq:upper_growth}
 \IP\big[ - X_1(t) \leq  h(t)
   \text{ eventually}\, \big] = 1.
\end{equation} 
\end{theorem}
 \begin{proof}
 First note that we can argue using 
 the stochastic domination from Lemma~\ref{lem:queue-monotonicity} that to prove~\eqref{eq:upper_growth} for the system started from $\eta(0) = u \in \bbDB$
 it is sufficient to  prove the same for the system started from $\eta (0 ) = \0$. Indeed, 
 recall from~\eqref{eq:queue-leftmost-finite} that for $\eta (t) \in \bbDB$ and $X_1(0)=0$,
 $- X_1(t) = \| \eta (t) \| - \| \eta (0)\|$. Thus under the coupling described in Lemma~\ref{lem:queue-monotonicity},
 we can build queueing processes $\eta^\0$ and $\eta^u$ and then construct corresponding particle systems with $- X_1^\0 (t) = \| \eta^\0 (t) \|$ and $ - X_1^u (t) = \| \eta^u (t) \| - \| u \| $, all on the same probability space, in such a way that $\| \eta^\0 (t) \| \leq \| \eta^u (t) \| \leq \| \eta^\0 (t) \| + \| u \|$ for all $t \geq 0$, since the total number of second class customers in the system (which count towards $\eta^u$ but not $\eta^\0$) starts from $\| u \|$ and is non-increasing. Thus the coupling gives $-X^u_1 (t) \leq - X^\0_1 (t)$ for all $t \geq 0$.

Thus it suffices to suppose that $\eta(0) = \0$.
 Consider configurations
\[
 G_t := \Big\{ u \in \bbD : \sum_{j=1}^{k_t} u_j
        \leq\ell(k_t)\Big\}.
\]
Lemma~\ref{lem:queue-monotonicity} shows that
$\IP_{\0} [ \eta (t) \in G_t ] \geq \IP_{\nu_\alpha} [ \eta(t) \in G_t]$, since
the queueing process under law~$\IP_{\nu_\alpha}$
dominates the process under law~$\IP_{\0}$.
Moreover, by hypothesis~\eqref{est_stat_measure_ell_k}, 
$\IP_{\nu_\alpha} [ \eta(t) \in G_t ]  \geq 1-\phi(k_t)$.
So we conclude that 
$\IP_\0 [ \eta(t) \in G_t ]  \geq 1-\phi(k_t)$.
Hence we obtain 
\begin{align}
    \label{eq:G-t-decompose}
    \IP_\0 [ | X_1 | > h (t) ] 
    & \leq \IP_\0 [ \eta (t) \notin G_t ] + \IP_\0 [ \{ \eta (t) \in G_t \} \cap \{ | X_1 | > h (t) \} ] \nonumber\\
  &  \leq \phi (k_t ) + \IP_\0 [ \{ \eta (t) \in G_t \} \cap \{ | X_1 | > h (t) \} ] .
\end{align}
For the last probability in~\eqref{eq:G-t-decompose}, suppose that both  $|X_1(t)|>h(t)$ and $\eta (t) \in G_t$. Then,  
since $\ell(k_t)\leq h(t)-1$, 
we  would have that $\eta_{j_0}(t) \geq 1$ 
for some~$j_0\geq k_t+1$, that is, at least
one customer went farther than~$k_t$ by time~$t$.
Then, by Proposition~2.3 of~\cite{MPW25},
it is straightforward to obtain that, for some $c>0$
(note that, with probability at least $1-\re^{-ct}$,
not more than 
$2a_1 t$
customers come to the system
before time~$t$)
\begin{align}
\label{eq:large-walk-bound}
\IP_\0 [ \{ \eta (t) \in G_t \} \cap \{ | X_1 | > h (t) \} ]
& \leq 
 \IP_\0 [\text{$\eta_j(t) \geq 1$ 
for some $j \geq k_t+1$}] \nonumber\\
& \leq c t \IP\Bigl[\max_{0\leq s\leq t}\zeta_s
   \geq k_t+1\Bigr] + \re^{-ct}.
\end{align}
Combining~\eqref{eq:G-t-decompose} and~\eqref{eq:large-walk-bound}, and using Fubini's theorem,
we obtain (here $| \, \cdot \, |$ denotes Lebesgue measure)
that, for some constant $C< \infty$, 
\[
\IE\big|\{t \geq 0: |X_1(t)|>h(t)\}\big|
\leq C + 
C \int_0^{\infty} \Bigl( \phi(k_t)
  + t \IP\big[\max_{0 \leq s\leq t}\zeta_s
   \geq k_t+1\big] \Bigr)\,\ud t. 
\]
By hypothesis~\eqref{eq:upper_growth_cond},
it follows that this quantity is finite, and so $|\{t \geq 0 : |X_1(t)|>h(t)\}|$
is a.s.\ finite. This does not automatically
imply that the set $\{t: |X_1(t)|>h(t)\}$ is a.s.\ 
bounded (because of continuous time), but let us 
instead show that the set $\{t: |X_1(t)|>h(t)+1\}$
must be a.s.\ bounded.
Recall that we assumed that $h'(t)\leq a_1$,
which means that $|X_1(t)|>h(t)+1$
implies that $|X_1(t)|>h(t+a_1^{-1})$.
Now, regardless of the past, with a uniformly
positive probability
the process~$|X_1|$
does not change its value on a time interval
of length~$a_1^{-1}$ (i.e., the leftmost particle
does not jump), meaning that 
if~$t_0\in\{t: |X_1(t)|>h(t)+1\}$ then
$[t_0,t_0+a_1^{-1}]\subset\{t: |X_1(t)|>h(t)\}$
with at least that probability. From this,
we obtain that the set $\{t: |X_1(t)|>h(t)+1\}$
must be a.s.\ bounded; indeed, if it were unbounded,
then, by the preceding argument, 
$|X_1|$ would exceed~$h$ on an a.s.\ infinite
sequence of non-intersecting intervals
of lengths~$a_1^{-1}$,
and so the expected
size of $\{t: |X_1(t)|>h(t)\}$ would be infinite.
This verifies~\eqref{eq:upper_growth} when $\eta (0) = \0$, and hence concludes the proof of Theorem~\ref{thm:upper_growth} as argued in the first paragraph of this proof.
\end{proof}

\subsection{Lamperti-type rates and proof of Theorem~\ref{thm:lamperti}}
\label{sec:Lamperti}

We note that Theorems~\ref{thm:MGinfty_comparison}
and~\ref{thm:upper_growth} can be used to deal with 
the ``dog and sheep'' example of Theorem~1.19 of~\cite{MPW25}. Rather than discussing this in detail,
we turn to the 
(somewhat more difficult) 
Lamperti-type rates example from Theorem~\ref{thm:lamperti}.
Thus we assume rates of the form~\eqref{eq:a-b-lamperti} where $0 < \mu  < 1/2$. 

Observe that, since $a_k < b_k$ for all $k \in \N$,
$\alpha_k < 1$ for all $k \in \N$, and so 
there is at least one admissible $\rho$, following
Remarks~\ref{rems:known-finite}\ref{rems:known-finite-i}.
Next, note that, by~\eqref{eq:alpha-k-def} and~\eqref{eq:a-b-lamperti},
\begin{align}
 \alpha_n &= \prod_{k=1}^n\frac{1-\frac{2\mu}{k}}{1+\frac{2\mu}{k}} = \prod_{k=1}^n \big(1-\tfrac{4\mu}{k}+O(k^{-2})\big)
 \nonumber\\
 &= \exp\Big(-\sum_{k=1}^n \big(\tfrac{4\mu}{k}+O(k^{-2})\big)\Big) \asymp n^{-4\mu}.
\label{eq:calc_alpha_n_Lamperti}
\end{align}
(Here and throughout the paper,
we write $f(n) \asymp g(n)$ to mean that there exists $c \in (1,\infty)$ for which $g(n)/c < f(n) < c g(n)$ for all but finitely many $n \in \N$.)
By~\eqref{eq:calc_alpha_n_Lamperti} and the fact $a_n \to 1/2$ as $n \to \infty$, from the classical classification for birth-death processes (see~\S\ref{sec:intro-customer} and~\cite[Ch.~8]{anderson}) we see that the customer random walk~$\zeta$ is in this case
 positive recurrent if $\mu > 1/4$, and null recurrent if $0 < \mu  \leq 1/4$.

\begin{proof}[Proof of Theorem~\ref{thm:lamperti}]
If $\mu > 1/4$, then, as explained above, $\zeta$ is positive recurrent and
Proposition~\ref{prop:walk-summary-known}\ref{prop:walk-summary-known-b}
gives  part~\ref{thm:lamperti-b} of the theorem.
It remains to prove part~\ref{thm:lamperti-a}; thus, suppose $0 <\mu  < 1/4$ (we comment on the critical case $\mu = 1/4$ in Remark~\ref{rem:lamperti-critical} below).
We are going to show that,  in this case, $X_1$
is transient, and obtain some estimates on the growth
of~$|X_1|$. 


Lemma~2.7.5 from~\cite{mpw-book}  implies
that, if starting at~$h\sqrt{t}$ with large 
enough~$h$, the customer's walk
will survive with at least a constant probability up 
to time~$t$.
Also,
a straightforward calculation similar 
to~\eqref{eq:calc_alpha_n_Lamperti} shows that
its scale function (recall~\eqref{eq:scale_func-def})
is
$f(x)\asymp x^{1+4\mu}$. 
By the Optional Stopping Theorem,
this means that the probability that a customer (starting at~$1$)
comes to~$m$ without hitting~$0$ is of order $m^{-(1+4\mu)}$.
Therefore, for large enough~$h$, 
we can write
(recall that, unless otherwise stated,
we assume that the walk~$\zeta$ starts at~$1$)
\begin{align}
\IP[\tau\geq t]
 &\geq \IP[\text{$\zeta$ goes to $h\sqrt{t}$
 before hitting~$0$, then survives till $t$}]
  \geq C_1 t^{-\frac{1}{2}(1+4\mu)}.
\label{tail_cust_lifetime_Lamperti}
\end{align}
for some $C_1>0$.
Then, for some $C_2>0$ we have
\[
 \IE(\tau\wedge t) = \int_0^t \IP[\tau\geq s]\, \ud s
  \geq C_2 t^{\frac{1}{2}-2\mu}.
\]
Consequently, using
Theorem~\ref{thm:MGinfty_comparison}\ref{thm:MGinfty_comparison-b}
for $0 < \mu< 1/4$ we will obtain, for small enough~$\eps_0$,
\begin{equation}
\label{growth_Lamp_lower}
|X_1(t)| \geq \eps_0 t^{\frac{1}{2}-2\mu},  
 \text{ eventually}, \as,
\end{equation}
thus, in particular, showing that~$X_1$ is transient. 

Then, with the help of Theorem~\ref{thm:upper_growth}
we obtain an upper bound for the growth
of~$|X_1(t)|$ in the case $0 < \mu < 1/4$. Namely, we
will now show that, for large enough~$C'$,
\begin{equation}
\label{Lamp_not_so_fast}
|X_1(t)| \leq C' (t\log t)^{\frac{1}{2}-2\mu} ,  
 \text{ eventually}, \as
\end{equation}
To apply Theorem~\ref{thm:upper_growth},
we first need to obtain a suitable large deviation
estimate for $u \in \bbD$ under $\nu_\alpha$,  
as in~\eqref{est_stat_measure_ell_k}.
Under $\nu_\alpha$, $u_1, \ldots, u_{k}$ are independent
random variables with $u_j \sim \geo{1-\alpha_j}$,
and the expectation of~$u_j$ is $\alpha_j/(1-\alpha_j) = O (j^{-4\mu})$ as $j \to \infty$.
Then, we do a standard calculation: first, recall that the 
moment generating function of $\geo{1-a}$ 
is $\frac{1-a}{1-a\re^\lambda}$, $\lambda<-\log a$.
Note the rates~\eqref{eq:a-b-lamperti} are such that $\sup_{j \in \N} \alpha_j = \alpha_1 < 1$; then, taking $\lambda\in (0,-\log \alpha_1)$ we have $\lambda < - \log \alpha_j$ for all $j$, and then, for $M \in \RP$,
\begin{align}
 \nu_\alpha \Bigl\{ u \in \bbDB :  \sum_{j=1}^k u_j > M k^{1-4\mu} \Bigr\} 
  &= 
 \nu_\alpha \Bigl\{ u \in \bbDB : \exp\Big(\lambda \sum_{j=1}^ku_j\Big) 
 > \exp\big(\lambda M k^{1-4\mu}\big)\Bigr\} 
\nonumber
\\
 &\leq \exp\Big(-\lambda M k^{1-4\mu}
  +\sum_{j=1}^k\log\frac{1-\alpha_j}{1-\alpha_j\re^\lambda}\Big).
\nonumber
\end{align}
Here it holds that, for all $\lambda >0$ and all $j \in \N$,
\[
\log\frac{1-\alpha_j}{1-\alpha_j\re^\lambda} 
= \log \left( 1 + \frac{\alpha_j(\re^\lambda -1)}{1-\alpha_j\re^\lambda} \right) \leq
\frac{\alpha_j(\re^\lambda -1)}{1-\alpha_j\re^\lambda},
\]
and so (recall $\alpha_j \leq \alpha_1 < 1$) there are constants $C < \infty$ and $\lambda \in (0, -\log \alpha_1)$ such that 
\[ \nu_\alpha \Bigl\{ u \in \bbDB :  \sum_{j=1}^k u_j > M k^{1-4\mu} \Bigr\} 
  \leq \exp\Big(-\lambda M k^{1-4\mu}
  + C \lambda \sum_{j=1}^k \alpha_j \Big).
\]
It follows from~\eqref{eq:calc_alpha_n_Lamperti} that
we can choose $M$ large enough such that, for some $c>0$,
\begin{equation}
     \nu_\alpha \Bigl\{ u \in \bbDB : \sum_{j=1}^k u_j > M k^{1-4\mu}\Bigr\} 
\leq \exp(-c k^{1-4\mu}).
\label{LD_sum_etas_Lamperti}
\end{equation}
Take $h(t)=C''(t\log t)^{\frac{1}{2}-2\mu}$,
and note that, with $\ell(k)=Mk^{1-4\mu}$,
we have $k_t = C_1 (t \log t)^{1/2}$,
with large~$C_1$.
Then we note that, 
 dominating the customer random walk with the symmetric simple random walk $S = (S_t)_{t \geq 0}$ with $S_0 =0$ and jumps of size to $+1$ and $-1$ each at rate $1/2$, and e.g.\ Proposition~2.1.2(b) of~\cite{LawLim10} for the discrete-time 
 chain and Poisson large deviations bounds,
we have with a large~$C_2$,
\begin{align*}
 \IP\big[\max_{0 \leq s\leq t}\zeta_s
   \geq k_t+1\big] & 
   \leq \IP\big[\max_{0 \leq s\leq t} S_s
   \geq k_t+1\big] \leq \exp(-C_2 \log t),
\end{align*}
where we can make $C_2 \in (0,\infty)$ as large as we like by choosing $C_1 \in (0,\infty)$ appropriately. 
Then, Theorem~\ref{thm:upper_growth} applies
and we obtain~\eqref{Lamp_not_so_fast}.
\end{proof}

\begin{remark}
    \label{rem:lamperti-critical}
It is not at all clear to us what to expect in the critical case
$\mu=1/4$.
As discussed following~\eqref{eq:calc_alpha_n_Lamperti}, the customer
random walk $\zeta$ is null recurrent.
Even to prove transience of $X_1$, we 
would have to do quite a fine analysis
of the distribution of~$\tau$:
it is clear that we will have $\IP [ \tau>t ]\sim C/t$, 
but, to apply Theorem~\ref{thm:MGinfty_comparison},
one would need to know the value of~$C$
(or at least show that $C>1$).
What is somewhat troubling, is that this
constant would change if we
modify the transition probabilities in finitely
many places, so it seems to be quite subtle indeed.
In any case, even if for this concrete model 
(with $\mu=1/4$) defined here the motion of the leftmost particle proves to be transient, the question remains: what if we further modify 
the customer random walk, by introducing a suitable
$O(\frac{1}{k\log k})$ correction into the transition
probabilities, making it more critically null-recurrent? For now, it is unclear to us 
if it is possible to make $|X_1|$ (null) recurrent as well in this way
(we refer to Theorem~\ref{thm:null-recurrent} below for an example where we can verify null recurrence).
As mentioned in~\S\ref{sec:MGinfinity},
one might even ask 
if an ``intermediate regime'' (i.e., neither recurrent 
nor transient) is possible (similarly to the case 
of the $M/G/\infty$ in~\cite{P25}).
\end{remark}

\subsection{An example with a null-recurrent leftmost particle}
\label{sec:null_recurrent}

The goal of this section is to demonstrate an example 
of rates satisfying Condition~\eqref{ass:rates} that admits an admissible solution,
starts from
$X_1 (0) = 0$ and $\eta(0) \in \bbDB$,  
and where
$X_1$ is null recurrent; see Theorem~\ref{thm:null-recurrent} below for the precise statement. 
Proposition~\ref{prop:walk-summary-known} shows that 
the customer random walk must itself be null recurrent to find such an example. Moreover,  (see~\eqref{eq:queue-leftmost-finite}) recurrence of $X_1$ will follow if we can establish recurrence of $\eta$ on $\bbDB$.

Let us consider a very rapidly growing sequence $(w_n)_{n \in \N}$,
defined by~$w_1=2$, $w_{n+1}=w_n^7$ (so that $w_n=2^{7^{n-1}}$).
For $k\geq 1$, denote $r_k=2(w_1+\cdots+w_{k-1})+w_k+1$ 
and $h_k=2(w_1+\cdots+w_k)+1$ (that is, we have
$h_k=r_k+w_k$ and $r_{k+1}=h_k+w_{k+1}$).
We then set 
\begin{equation}
\label{eq:null-recurrent-example}
(a_k,b_k) = 
 \begin{cases}
    (1,\re), & k=1,\ldots,r_1,\\
    (\re,1), & k=r_1+1,\ldots, h_1,\\
    (1,\re), & k=h_1+1,\ldots, r_2,\\
    (\re,1), & k=r_2+1,\ldots,h_2,\\
    \text{and so on.}
 \end{cases}
\end{equation}
Here is the main result of this section.

\begin{theorem}
    \label{thm:null-recurrent}
    Consider the example with rates given by~\eqref{eq:null-recurrent-example},
    and suppose that $X_1 (0) = 0$ and $\eta(0) \in \bbDB$. Then $X_1$ is null recurrent, i.e.,  (i) $X_1 (t) \to -\infty$ in probability, but (ii) $\{ t \geq 0 : X_1(t) = 0\}$ is unbounded, a.s.
\end{theorem}
\begin{proof}
It is straightforward to obtain that $\rho=\alpha$
is admissible; indeed, we have that (denoting 
also $h_0:=1$), for $j\in \N$, 
\[
\alpha_k   =
 \begin{cases}
  \re^{-k+h_{j-1}-1}, & \text{for }k\in[h_{j-1},r_j],\\
  \re^{-w_j-1 + (k-r_j)}, & \text{for }k\in[r_j+1,h_j],
 \end{cases}
\]
so that, in particular, $\alpha_k\leq \re^{-1}$ for all~$k \in \N$.
Also, we observe that $\alpha_{h_j}=\re^{-1}$ for all~$j \in \ZP$,
meaning that $\sum_{k=1}^\infty \alpha_k = \infty$;
that is, we already know that $X_1$ cannot be 
positive recurrent. Indeed, Corollary~1.16 of~\cite{MPW25} shows that $X_1 (t) \to -\infty$ in probability, as claimed in part~(i) of the theorem. The rest of the proof is devoted to establishing part~(ii).

First we explain the intuition for the construction. The transition rates were chosen in such a way
that in each of the intervals~$[r_k,r_{k+1}]$
the customer's walk has ``drift inside'' (directed towards~$h_k$),
which reminds us
of the so-called potential wells, the notion frequently
used when studying random walk in one-dimensional random environments, see e.g.~\cite{FGP10}.
The general idea of this example 
is that a customer needs time at least of order $\re^{w_n}$
to go out of a potential well
on the ``scale''~$n$, but
time $\re^{w_{n+1}}\gg \re^{w_n}$ is needed for the system 
to ``find-and-explore''
the next well; so, hopefully, before the system manages
to ``advance'', many instances with empty queues will occur 
with a very high probability. 

So, when starting from a finite configuration,
we are dealing with the countable Markov chain $\eta$ on the state space~$\bbDB$. 
Recall that $\nu_\alpha$ defined in~\eqref{eq:df_measure_rho}
is a stationary and reversible
measure for $\eta$. Therefore the Markov chain $\eta$
can be represented as an electric network: there is an edge between
two configurations~$u,u'\in\bbDB$ if it is possible to obtain~$u'$
 from~$u$ in just one transition (i.e., a customer going from one queue
to a neighbouring one, or a customer leaving the system, or a new customer 
arriving to the first queue). The corresponding \emph{conductances}
are then defined in a natural way: for an
(un-oriented) edge $\eps=(u,u')$,
we have $c(\eps)=\nu_\alpha(u)\lambda(u,u')$, where
$\lambda(u,u')\in\{a_k,b_k, k\geq 1\}$ is the rate of the 
corresponding transition. Note that $1 \leq \lambda (\eps) \leq \re$ for all edges~$\eps$.
Also, it is natural to regard the empty configuration 
as the ``origin'' of~$\bbDB$. 

We call a set of edges a \emph{cut-set}
if every infinite self-avoiding path starting at the origin has 
to pass through that set. We intend to use the result
of Nash-Williams~\cite{NW59}
for proving the recurrence: it says that if there is a sequence
of non-intersecting cut-sets $(\Pi_n)_{n \in \N}$ such
that 
\begin{equation}
\label{Nash-Williams}
 \sum_{n\in \N} \Big(\sum_{\eps\in \Pi_n}c(\eps)\Big)^{-1} = \infty,
\end{equation}
then the Markov chain is recurrent.

Let us define the \emph{weight}  
of a configuration~$u \in \bbD$
as $\cW(u) = \sum_{k\in \N}k u_k$ (i.e., a customer
in the first queue weighs one unit, a customer in the second queue
weighs two units, and so on). An important observation is that 
every transition changes (increases or decreases) the weight by
exactly one unit, so that the edge set of the graph 
is $\big\{(u,u'): u,u' \in \bbDB, \, |\cW(u)-\cW(u')|=1\big\}$.
For $n\in \ZP$, let us denote
\[
 \Delta_n := \big\{u\in\bbDB : \cW(u)=n\big\}.
\]
It is important to observe that the cardinality of~$\Delta_n$
is the so-called \emph{partition function} of~$n$
(i.e., the number of possible partitions of~$n$
into a sum of positive integer terms);
indeed, a customer at the $k$th queue plays the role
of a term~$k$ in the partition of~$n$. 
A lot is known about the asymptotic behaviour of the partition
function; we, however, will only need the following
fact (see e.g.~\cite{HarRam18}): there is $\gamma_1 \in \RP$
such that 
\begin{equation}
\label{HarRam_estimate}
 |\Delta_n| \leq \exp\big(\gamma_1 n^{1/2}\big) ,
 \text{ for all }n\in \N.
\end{equation}
Then, define the sequence of cut-sets
\[
 \Pi_n = \big\{(u,u'): u\in \Delta_n, u'\in \Delta_{n+1}\big\}.
\]
We intend to prove that, for some $\gamma_2>0$ 
and all~$j \geq 1$
(in the following, note that~$r_j$ and~$w_j$ 
are asymptotically equivalent 
in the sense that $r_j/w_j \to 1$ as $j\to\infty$)
\begin{equation}
\label{w1/6diverges}
 \sum_{\eps \in \Pi_{r_j}}c(\eps) \leq \exp\big(-\gamma_2 w_{j}^{1/6}\big) .
\end{equation}
This will already be enough to prove the recurrence, as it would
show that the series in~\eqref{Nash-Williams} contains
a sub-series with terms not converging to zero (even unbounded).
Since a configuration $u\in \Delta_n$ has at most~$O(\sqrt{n})$
non-empty queues (so at most~$O(\sqrt{n})$ edges connected to it) each with $c(\eps) \leq \re \nu_\alpha (u)$,
it holds that 
\begin{equation}
\label{estimate_wrj1}
 \sum_{\eps\in \Pi_{r_j}}c(\eps) 
  \leq \gamma_3 \sum_{u \in \Delta_{r_j}}\nu_\alpha(u)\sqrt{w_{j}}.
\end{equation}
Also, note from~\eqref{eq:df_measure_rho} that $\nu_\alpha (u) \leq \alpha_k^{u_k}$ for every $k \in \N$, and since $\alpha_k \leq \re^{-1}$, $\nu_\alpha (u) \leq \re^{-u_k}$ for every $k \in \N$.
Define $F_j := \{ u \in \Delta_{r_j} : \max_{k \in \N} u_k \geq w_j^{2/3}\}$, and observe that if $u \in F_j$,
 then 
$\nu_\alpha(u)\leq \re^{-w_j^{2/3}}$. Suppose instead that $u \in \Delta_{r_j} \setminus F_j$, so that   $u_k < w_j^{2/3}$
for all~$k \in \N$. Notice that, for $k\geq h_{j-1}$, we have
$\log \alpha_k = -(k+1)+h_{j-1}$;
 one can easily obtain that (at least for large enough~$j$)
 $k+1-h_{j-1}\geq k/2$
for $k\geq 3w_j^{1/7}$. Also, since 
\[
 \sum_{k=1}^{3w_j^{1/7}} k w_j^{2/3} 
  = O\big(w_j^{\frac{2}{7}+\frac{2}{3}}\big)
  = O\big(w_j^{\frac{20}{21}}\big),
\]
we have (for large enough~$j$ and for~$u$ such that
$u_k < w_j^{2/3}$ for all~$k$) 
\[
 \sum_{k\geq 3w_j^{1/7}}k u_k \geq \frac{r_j}{2}.
\]
So, with the above to hand, we have for $u \in \Delta_{r_j} \setminus F_j$, 
\[
 \nu_\alpha(u) 
\leq \exp\Big(\sum_{k\geq 1} u_k \log\alpha_k\Big) 
  \leq \exp\Big(-\sum_{k\geq 3w_j^{1/7}}\frac{k u_k}{2}\Big) 
  \leq \exp\Big(-\frac{r_j}{4}\Big) .
\]
Combined with the case of $u \in F_j$, we conclude that,
for all $u \in \Delta_{r_j}$, for all $j$ large enough,  
$\nu_\alpha(u)\leq \re^{-\gamma_4 w_j^{2/3}}$.
Then, \eqref{estimate_wrj1} and~\eqref{HarRam_estimate}
imply~\eqref{w1/6diverges}, so (as we argued before) 
the  Markov chain $\eta$ (and hence $X_1$)
is recurrent.
\end{proof}

\section{Stationary initial configurations: dynamic recurrence}
\label{sec:stationarity-start}

In this section, we 
consider the finer dynamics of the leftmost particle started from a stationary measure (assuming there is one, of which there might be several).
We know (Proposition~1.6 of~\cite{MPW25}) that if $\eta(0)$
is started from the stationary measure $\nu_\rho$ defined by~\eqref{eq:df_measure_rho} for an admissible $\rho = \rho(v)$, then $X_1(t) / t \to v$ as $t \to \infty$. The next result shows a sort of \emph{dynamic recurrence},
meaning that the particle oscillates each side of the strong-law behaviour.

\begin{theorem}
\label{thm:rec_stationary}
Suppose that Condition~\eqref{ass:rates} holds and
that $\rho = \rho (v)$ is admissible for a given $v \geq v_0$. Take $\eta (0) \sim \nu_\rho$ and $X_1(0)=0$.
Then 
\begin{equation}
\label{eq:rec_stat_crossings}
 \IP\big[\text{the set }
\{ t \geq 0 : X_1(t)=vt\} \text{ is unbounded}\big]=1.
\end{equation}
\end{theorem}
\begin{remark}
\label{rem:importance-of-initial-condition}
As mentioned in Remarks~\ref{rems:lamperti}\ref{rems:lamperti-v},
there are cases where $v_0 =0$ in which 
we have transience for $X_1$ started from $\eta(0) \in \bbXB$,
but for which Theorem~\ref{thm:rec_stationary} shows recurrence started from $\eta(0) \sim \nu_\rho$ corresponding to the minimal $\rho = \rho (v_0 ) = \alpha$. Intuitively, in such situations $\nu_\rho$ (which is, necessarily in that case, supported on configurations of infinitely many empty sites to the right of the leftmost particle) leaves enough space to reduce the leftwards pressure on the leftmost particle.
\end{remark}

\begin{proof}[Proof
of Theorem~\ref{thm:rec_stationary}]
Suppose that $\rho = \rho(v)$ is admissible, and take $\eta(0) \sim \nu_\rho$. 
For technical reasons, we treat the cases $v\neq 0$
and $v=0$ separately. First, assume that $v\neq 0$.
Then, it is enough to prove that
\begin{equation}
\label{eq:rec_stat}
 \IP\big[\text{for any }t_0>0 \text{ there exists }
 t'>t_0 \text{ such that }|X_1(t')-vt'|\leq 1\big]=1.
\end{equation}
Indeed, let us show that~\eqref{eq:rec_stat}
implies~\eqref{eq:rec_stat_crossings}.
Indeed, assume first that $v<0$.
If we have $X_1(t')\in [vt'-1, vt')$,
then with probability bounded away from zero
(and independently from the past)
the leftmost particle will not move in the time 
interval $[t',t'+|v|^{-1}]$, 
mean that, by time $t'' = t'+|v|^{-1}$ we will have
$X_1(t'') = X_1 (t') \in [ vt'' , vt''+1)$, and hence, by continuity of $t \mapsto X_1 (t) - vt$ for $t \in [t',t'']$, 
$X_1(t) = vt$ for some $t \in [t',t'']$.
On the other hand, suppose 
$X_1(t')\in (vt', vt'+1]$. Then with probability bounded away from zero
the leftmost particle will make exactly two jumps, both to the left, during time interval $(t', t'+|v|^{-1})$; then at time $t'' = t'+ |v|^{-1}$ we have
$X_1(t'') = X_1 (t') - 2 \in (vt'' -1 , vt'']$, and so either $X_1(t'') = vt''$ or else we can repeat the first argument.
In other words, on $\{ | X_1 (t') - vt ' | \leq 1 \}$, we have $\IP [ X_1 (t) = v t \text{ for some } t \in [t',t'+2 |v|^{-1}] \mid \cF_{t'} ] \geq \eps$ for some $\eps >0$ depending only on~$a_1$ and~$v$. Hence L\'evy's conditional Borel--Cantelli lemma shows that~\eqref{eq:rec_stat} implies~\eqref{eq:rec_stat_crossings}.

To prove~\eqref{eq:rec_stat}, it is clearly enough
to prove the following: for arbitrary~$n_0\in\N$,
we have
\begin{equation}
\label{eq:hit_0_after_n0}
 \IP\big[\text{there exists }
 t'\geq n_0 \text{ such that }|X_1(t')-vt'|\leq 1\big]=1.
\end{equation}
For the rest of the proof, let~$n_0$ be fixed.

 Let $N_1(t)$ (respectively, $N_2(t)$)
 be the number of jumps of the leftmost particle 
to the left (respectively, to the right) up to time~$t$.
Clearly, $N_1$ is a Poisson process of rate~$a_1$.
As for~$N_2$, due to~Proposition~1.6 of~\cite{MPW25}
(as noted in the proof there,
the original exit flow is 
the input flow of the reverse process; 
this is analogous to Burke's theorem~\cite{burke}),
we have that, in stationarity, $N_2$ is a Poisson process of rate~$a_1+v$.
For $t\geq 0$   define $Y(t) := N_1 (t) - N_2(t) + vt$, 
and note that (since $X_1 (0) = 0$) we have $| X_1 (t) - vt | = |Y(t)|$. 
The Poisson processes $N_1$ and $N_2$ are not, generally, independent, but nevertheless applying
the strong law of large numbers for the two processes shows that $\lim_{t \to \infty} (Y(t)/t) = 0$, a.s.

Consider the events (illustrated in Figure~\ref{f_remain_strip})
\begin{align*}
    E^{(1)}(t) := \big\{ Y(t+s) \geq Y(t) +1 
  \text{ for all }s\geq n_0\big\},
  \\ 
      E^{(2)}(t) := \big\{ Y(t+s) \leq Y(t) - 1
  \text{ for all }s\geq n_0\big\}.
\end{align*}
\begin{figure}
\begin{center}
\includegraphics[width=0.75\textwidth]{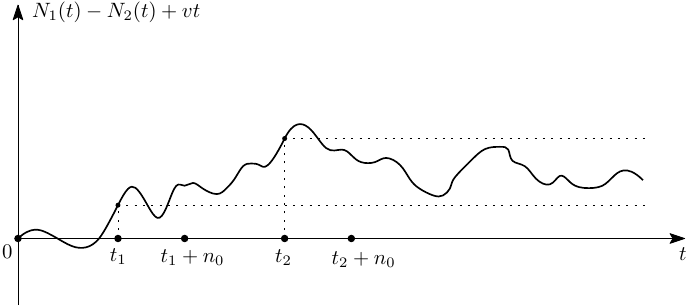}
\caption{The events $E^{(1)}(t_1)$ and $E^{(2)}(t_2)$
(with $t_1<t_2$).}
\label{f_remain_strip}
\end{center}
\end{figure}
Let $\psi_i :=\IP[E^{(i)}(0)]$, and note
that $\IP [ E^{(i)} (t) ] = \psi_i$ does not depend on $t \geq 0$.
We claim that $\psi_1=\psi_2=0$, from which we
immediately obtain~\eqref{eq:hit_0_after_n0}; we will next justify the claim.

To verify the claim that $\psi_1 = 0$, we suppose, for a contradiction, that $\psi_1 > 0$.
Abbreviate $Z_k:=\2 {E^{(1)} (k)}$, and consider the sequence 
$Z=(Z_0,Z_1,Z_2,\ldots)$;
we will show that  the sequence~$Z$ 
is ergodic (i.e., $\lim_{n \to \infty} n^{-1} \sum_{k=1}^n 
Z_k \to \psi_1$, a.s.).

First,  since we have stationary initial condition $\eta (0) \sim \nu_\rho$,
 the sequence~$Z$ 
 is stationary, meaning that, for each $n \in \N$, 
$Z$ 
and~$\theta^n Z$ 
have the same distribution, where~$\theta$ is the unit time-shift operator, i.e., 
$\theta Z := (Z_1,Z_2,Z_3,\ldots)$.
Define the $\sigma$-algebra of invariant events by $\cI:=\{A\in\sigma(Z_0,Z_1,\ldots):\theta^{-1}A=A\}$.
By Birkhoff's ergodic theorem (see e.g.~\cite[p.~561]{kallenberg}), to show that the sequence~$Z$
is ergodic 
it suffices to show that $\cI$ is trivial, i.e., $\IP (A) \in \{0,1\}$
for every $A \in \cI$. 
 
Next observe that, 
 in the stationary regime, every queue becomes empty on an a.s.~unbounded set of times.  
 To see this, fix $K \in \N$ and consider a fixed queue. Each time the queue has~$K$ customers, there is a positive chance it will be empty after a fixed time interval (uniformly in time and in the number of customers in the other queues). Hence if the queue length is~$K$ on an unbounded set of times, a.s.~the queue length will be~$0$ on an unbounded set of times. On the other hand, if, for each~$K \in \N$, the set of times when the queue has~$K$ customers has a finite supremum, then the queue size tends to infinity. But if the latter has a positive probability, the expected size of a queue cannot remain constant, contradicting stationarity. 
 This verifies the claim that every queue is empty on an unbounded set of times. Hence every customer in the system will eventually be served at its current queue, and hence every customer in the system will complete every step of its associated realization of the customer walk eventually.

To argue triviality, we will identify a probabilistic structure that allows us to appeal to Kolmogorov's $0$--$1$ law, and to do so we use
 a different formal construction of the process than the Harris graphical construction which underpins~\cite{MPW25}. Instead, to each particle, attach the following attributes:
 \begin{itemize}
     \item the time it appears in the system (which is 0 for those initially present there); 
     \item the skeleton walk it does (which also contains the information about its initial position in case it was initially present in the system);
     \item  for customer $k$, the sequence of Poisson clocks $(t^{(k)}_{j,n}, j\in \N, n\in \N)$, with intensities depending on its skeleton walk, where $n$ is the number of the jump and $(t^{(k)}_{1,n}, t^{(k)}_{2,n}, \ldots )$ are the attempted jump times (the jumps is only executed if the customer is highest priority in its queue, according to the priority policy given below).
 \end{itemize}
Declare a customer to have priority $m \in \N$ if it was initially at the $m$th queue, or else arrived to the system during time interval $(m-1,m]$, and suppose that
the service policy is such that priority $m$ customers have priority of service over all customers of priority $\ell < m$.
Also choose some order (by any reasonable rule) among the customers of the same priority index. Then, the above works as a formal construction of the process, and it is also true that behaviour of a customer with priority $m$ is not affected
by any customers of priority $\ell < m$.

Now, we have two sequences of independent random elements:
\begin{itemize}
    \item  particles initially at $m$th queue (with all their attributes), for $m \in \N$, and
    \item particles which arrived to the system in $(m-1,m]$ (again, with all their attributes), for $m \in \N$.
\end{itemize}
Consider first the case of when the customer random walk $\zeta$ is recurrent. We argue that any invariant event $A \in \cI$ is also a tail event with respect to the above i.i.d.~sequence, equivalently, every $A \in \cI$ is independent of customers of any finite priority, and so has probability $0$ or $1$ by Kolmogorov's law. But, for any $m \in \N$ and any $\omega$ (i.e., the collection of attributes of all the customers) there is a time shift that  eliminates  all customers of priority $\ell < m$, and the (future and past) evolution of customers of priority at least $m$ does not depend on these. Hence customers of any finite priority cannot influence the occurrence of $A \in \cI$. This demonstrates that every $A \in \cI$ is a tail event, and then the Kolmogorov $0$--$1$ law shows that $\IP (A) \in \{0,1\}$. 

On the other hand, suppose the customer random walk $\zeta$ is transient.  Consider $\cI_N$, the class of invariant events that are
measurable with respect to $\sigma ( \eta_k (t) , t \geq 0, 1 \leq k \leq N)$. Then $\cI_N$ is  a $\pi$-system and $\sigma ( \cI_1, \cI_2, \ldots ) = \cI$, so, by Dynkin's $\pi$--$\lambda$ theorem
(note that all $0$--$1$ events form a $\lambda$-system),
to prove that~$\cI$ is trivial it is sufficient to prove that~$\cI_N$ is trivial for each $N \in \N$. Fix such~$N$. Now every customer of priority $1, \ldots, N$ will either leave the system in finite time (by exiting via the leftmost queue, as in the recurrent case) or else will eventually never return to a queue of index in $1, \ldots, N$ (by transience); in either case, these customers cannot influence events $A \in \cI_N$, and so, by a similar argument to before,
for every $N \in \N$ and every $A \in \cI_N$, $A$ is a tail event. By the $\pi$--$\lambda$ argument above,
this verifies that  $\IP (A) \in \{0,1\}$ for every $A \in \cI$. This completes the proof of ergodicity.

Having established that the sequence~$Z$ 
is ergodic,  the hypothesis that $\psi_1 >0$,
implies that (asymptotically) a proportion~$\psi_1$
of events $(E^{(1)}(0), E^{(1)}(1), E^{(1)}(2),\ldots)$
will occur; moreover, we can find a (random) sequence
$m_1<m_2<m_3<\cdots$ such that
$E^{(1)}(m_i)$ occur and $m_{i+1}-m_i\geq n_0$ 
for all~$i$. The density of that sequence will be still
positive, at least $\frac{\psi_1}{n_0}$.


Note that if $G^{(1)} (m_k)$ occurs, then, since $m_{k+1} - m_k \geq n_0$, 
$Y( m_{k+1} ) = Y ({m_k} + (m_{k+1} -m_k) ) \geq Y(m_k) +1$.
It follows that, 
for every $k \in \N$, using that $G(m_1), \ldots, G(m_k)$ occur, we get
$Y(m_k) \geq Y(m_1) +k$.
As already observed, if $\psi_1 >0$ then there exists $\eps >0$ such that
$\limsup_{k \to \infty} (k/m_k) \geq \eps$, a.s.
It follows that
\[ \limsup_{t \to \infty} \frac{Y(t)}{t} \geq \eps, \as .
\]
But the strong law of large numbers said that $Y(t)/t \to 0$, a.s., giving a contradiction.
Thus it must be that $\psi_1 = 0$.
A similar argument shows that $\psi_2 =0$, and thus verifies~\eqref{eq:hit_0_after_n0} in the case $v \neq 0$.


We now briefly comment on the case $v=0$.
In this case, it is not immediate to obtain that~\eqref{eq:rec_stat}
implies~\eqref{eq:rec_stat_crossings}.
On the other hand, since $X_1(t)-vt$ now only assumes
integer values, a  more direct  argument goes through;
namely, for $t\geq 0$ and $i \in \{ 1,2 \}$, define the events
\[
 E^{(i)}(t) = \{N_i(s+t)-N_i(t)> N_{3-i}(s+t)-N_{3-i}(t)
  \text{ for all }s\geq 
  n_0\},
\]
and proceed analogously. We omit the details.
\end{proof}

\section*{Acknowledgements}
\addcontentsline{toc}{section}{Acknowledgements}

The work of MM and AW was supported by EPSRC grant EP/W00657X/1.
SP was partially supported by
CMUP, member of LASI, which is financed by national funds
through FCT (Funda\c{c}\~ao
para a Ci\^encia e a Tecnologia, I.P.) 
under the project with reference UID/00144/2025,
\url{https://doi.org/10.54499/UID/00144/2025}.

\printbibliography

@article {andjel82,
    AUTHOR = {Andjel, E. D.},
     TITLE = {Invariant measures for the zero range processes},
   JOURNAL = {Ann. Probab.},
  FJOURNAL = {The Annals of Probability},
    VOLUME = {10},
      YEAR = {1982},
    NUMBER = {3},
     PAGES = {525--547},
      DOI = {10.1214/aop/1176993765},
}

@book {anderson,
    AUTHOR = {Anderson, W. J.},
     TITLE = {Continuous-time {M}arkov Chains},
    SERIES = {Springer Series in Statistics: Probability and its Applications},
 PUBLISHER = {Springer-Verlag},
     ADDRESS = {New York}, 
      YEAR = {1991},
     PAGES = {xii+355},
      ISBN = {0-387-97369-9},
       DOI = {10.1007/978-1-4612-3038-0},
}

@book {kallenberg,
    AUTHOR = {Kallenberg, Olav},
     TITLE = {Foundations of modern probability},
    SERIES = {Probability Theory and Stochastic Modelling},
    VOLUME = {99},
   EDITION = {Third},
 PUBLISHER = {Springer, Cham},
      YEAR = {[2021] \copyright 2021},
     PAGES = {xii+946},
      ISBN = {978-3-030-61871-1; 978-3-030-61870-4},
   MRCLASS = {60-01 (60A10 60G05)},
  MRNUMBER = {4226142},
MRREVIEWER = {Myron\ Hlynka},
       DOI = {10.1007/978-3-030-61871-1},
       URL = {https://doi.org/10.1007/978-3-030-61871-1},
}

@article {lamp1,
    AUTHOR = {Lamperti, J.},
     TITLE = {Criteria for the recurrence or transience of stochastic
              process. {I}},
   JOURNAL = {J. Math. Anal. Appl.},
  FJOURNAL = {Journal of Mathematical Analysis and Applications},
    VOLUME = {1},
      YEAR = {1960},
     PAGES = {314--330},
      ISSN = {0022-247X},
   MRCLASS = {60.40},
  MRNUMBER = {126872},
MRREVIEWER = {T.\ E.\ Harris},
       DOI = {10.1016/0022-247X(60)90005-6},
       URL = {https://doi.org/10.1016/0022-247X(60)90005-6},
}

@article {lamp2,
    AUTHOR = {Lamperti, J.},
     TITLE = {Criteria for stochastic processes. {II}. {P}assage-time
              moments},
   JOURNAL = {J. Math. Anal. Appl.},
  FJOURNAL = {Journal of Mathematical Analysis and Applications},
    VOLUME = {7},
      YEAR = {1963},
     PAGES = {127--145},
      ISSN = {0022-247X},
   MRCLASS = {60.60},
  MRNUMBER = {159361},
MRREVIEWER = {K.\ L.\ Chung},
       DOI = {10.1016/0022-247X(63)90083-0},
       URL = {https://doi.org/10.1016/0022-247X(63)90083-0},
}

@article {arratia83,
    AUTHOR = {Arratia, R.},
     TITLE = {The motion of a tagged particle in the simple symmetric
              exclusion system on {${\bf Z}$}},
   JOURNAL = {Ann. Probab.},
  FJOURNAL = {The Annals of Probability},
    VOLUME = {11},
      YEAR = {1983},
    NUMBER = {2},
     PAGES = {362--373},
     DOI = {10.1214/aop/1176993602},
}

@article {bmrs2017,
    AUTHOR = {Bahadoran, C. and Mountford, T. and Ravishankar, K. and Saada,
              E.},
     TITLE = {Supercritical behavior of asymmetric zero-range process with
              sitewise disorder},
   JOURNAL = {Ann. Inst. Henri Poincar\'e{} Probab. Stat.},
  FJOURNAL = {Annales de l'Institut Henri Poincar\'e{} Probabilit\'es et
              Statistiques},
    VOLUME = {53},
      YEAR = {2017},
    NUMBER = {2},
     PAGES = {766--801},
       DOI = {10.1214/15-AIHP736},
}

@article {burke,
    AUTHOR = {Burke, P. J.},
     TITLE = {The output of a queuing system},
   JOURNAL = {Operations Res.},
  FJOURNAL = {Operations Research},
    VOLUME = {4},
      YEAR = {1956},
     PAGES = {699--704 (1957)},
      not-ISSN = {0030-364X,1526-5463},
   MRCLASS = {90.0X},
  MRNUMBER = {83416},
MRREVIEWER = {E.\ Reich},
       DOI = {10.1287/opre.4.6.699},
       not-URL = {https://doi.org/10.1287/opre.4.6.699},
}

@article {HarRam18,
    AUTHOR = {Hardy, G. H. and Ramanujan, S.},
     TITLE = {Asymptotic formul\ae\ in combinatory analysis},
   JOURNAL = {Proc. London Math. Soc. (2)},
  FJOURNAL = {Proceedings of the London Mathematical Society. Second Series},
    VOLUME = {17},
      YEAR = {1918},
     PAGES = {75--115},
      not-ISSN = {0024-6115},
   MRCLASS = {99-04},
  MRNUMBER = {1575586},
       DOI = {10.1112/plms/s2-17.1.75},
       not-URL = {https://doi.org/10.1112/plms/s2-17.1.75},
}

@article {NW59,
    AUTHOR = {Nash-Williams, C. {\relax St.}\ J. A.},
     TITLE = {Random walk and electric currents in networks},
   JOURNAL = {Proc. Cambridge Philos. Soc.},
  FJOURNAL = {Proceedings of the Cambridge Philosophical Society},
    VOLUME = {55},
      YEAR = {1959},
     PAGES = {181--194},
      not-ISSN = {0008-1981},
   MRCLASS = {60.65},
  MRNUMBER = {124932},
MRREVIEWER = {D.\ G.\ Kendall},
       DOI = {10.1017/s0305004100033879},
       not-URL = {https://doi.org/10.1017/s0305004100033879},
}

@article {P25,
    AUTHOR = {Popov, S.},
     TITLE = {On transience of {${\rm M}/{\rm G}/\infty$} queues},
   JOURNAL = {J. Appl. Probab.},
  FJOURNAL = {Journal of Applied Probability},
    VOLUME = {62},
      YEAR = {2025},
    NUMBER = {2},
     PAGES = {572--575},
      not-ISSN = {0021-9002,1475-6072},
   MRCLASS = {60K25 (60G55)},
  MRNUMBER = {4899548},
       DOI = {10.1017/jpr.2024.78},
       not-URL = {https://doi.org/10.1017/jpr.2024.78},
}

@article {FGP10,
    AUTHOR = {Fribergh, A. and Gantert, N. and Popov, S.},
     TITLE = {On slowdown and speedup of transient random walks in random
              environment},
   JOURNAL = {Probab. Theory Related Fields},
  FJOURNAL = {Probability Theory and Related Fields},
    VOLUME = {147},
      YEAR = {2010},
    NUMBER = {1-2},
     PAGES = {43--88},
      not-ISSN = {0178-8051,1432-2064},
   MRCLASS = {60K37 (60F10 60G50)},
  MRNUMBER = {2594347},
MRREVIEWER = {Marcel\ Ortgiese},
       DOI = {10.1007/s00440-009-0201-2},
       not-URL = {https://doi.org/10.1007/s00440-009-0201-2},
}

@article {bfl,
    AUTHOR = {Benjamini, I. and Ferrari, P. A. and Landim, C.},
     TITLE = {Asymmetric conservative processes with random rates},
   JOURNAL = {Stochastic Process. Appl.},
  FJOURNAL = {Stochastic Processes and their Applications},
    VOLUME = {61},
      YEAR = {1996},
    NUMBER = {2},
     PAGES = {181--204},
       DOI = {10.1016/0304-4149(95)00077-1},
}

@incollection {ferrari92,
    AUTHOR = {Ferrari, P. A.},
     TITLE = {Shocks in the {B}urgers equation and the asymmetric simple exclusion process},
 BOOKTITLE = {Statistical Physics, Automata Networks and Dynamical Systems ({S}antiago, 1990)},
    SERIES = {Math. Appl.},
     PAGES = {25--64},
 PUBLISHER = {Kluwer Acad. Publ.}, 
   ADDRESS = {Dordrecht},
      YEAR = {1992},
      ISBN = {0-7923-1595-2},
       DOI = {10.1007/978-94-011-2578-9_2},
}

@article {FF94,
    AUTHOR = {Ferrari, P. A. and Fontes, L. R. G.},
     TITLE = {The net output process of a system with infinitely many queues},
   JOURNAL = {Ann. Appl. Probab.},
  FJOURNAL = {The Annals of Applied Probability},
    VOLUME = {4},
      YEAR = {1994},
    NUMBER = {4},
     PAGES = {1129--1144},
      DOI = {10.1214/aoap/1177004907},
}

@article {kipnis,
    AUTHOR = {Kipnis, C.},
     TITLE = {Central limit theorems for infinite series of queues and applications to simple exclusion},
   JOURNAL = {Ann. Probab.},
  FJOURNAL = {The Annals of Probability},
    VOLUME = {14},
      YEAR = {1986},
    NUMBER = {2},
     PAGES = {397--408},
}

@book {LawLim10,
    AUTHOR = {Lawler, G. F. and Limic, V.},
     TITLE = {Random walk: a modern introduction},
    SERIES = {Cambridge Studies in Advanced Mathematics},
    VOLUME = {123},
 PUBLISHER = {Cambridge University Press, Cambridge},
      YEAR = {2010},
     PAGES = {xii+364},
      ISBN = {978-0-521-51918-2},
   MRCLASS = {60G50 (60-02)},
  MRNUMBER = {2677157 (2012a:60132)},
MRREVIEWER = {Andrew R. Wade},
}

@article {mmpw,
    AUTHOR = {Malyshev, V. and Menshikov, M. and Popov, S. and Wade, A.},
     TITLE = {Dynamics of finite inhomogeneous particle systems with exclusion interaction},
   JOURNAL = {J. Stat. Phys.},
  FJOURNAL = {Journal of Statistical Physics},
    VOLUME = {190},
      YEAR = {2023},
    NUMBER = {11},
     PAGES = {Paper No. 184, 32},
       DOI = {10.1007/s10955-023-03190-8},
}

@article {MPW25,
    AUTHOR = {Menshikov, M. and Popov, S. and Wade, A.},
     TITLE = {Semi-infinite particle systems with exclusion interaction and heterogeneous jump rates},
   JOURNAL = {Probab. Theory Relat. Fields},
      YEAR = {2025},
}

@book {mpw-book,
    AUTHOR = {Menshikov, M. and Popov, S. and Wade, A.},
     TITLE = {Non-homogeneous Random Walks},
    SERIES = {Cambridge Tracts in Mathematics},
    VOLUME = {209},
 PUBLISHER = {Cambridge University Press}, 
  ADDRESS = {Cambridge},
      YEAR = {2017},
     PAGES = {xviii+363},
      ISBN = {978-1-107-02669-8},
       DOI = {10.1017/9781139208468},
}

\end{document}

